\documentclass[reqno,11pt]{amsart}
\usepackage{amsmath,amssymb,latexsym,cancel,rotating}

\numberwithin{equation}{section}

\newtheorem{theorem}{Theorem}[section]
\newtheorem{proposition}[theorem]{Proposition}

\newtheorem{corollary}[theorem]{Corollary}
\newtheorem{lemma}[theorem]{Lemma}

\newtheorem{remark}[theorem]{Remark}

\input xy
\xyoption{poly}
\xyoption{2cell}
\xyoption{all}

\def\A{\mathcal{A}}

\def\NN{\Bbb N}

\def\ZZ{\mathbb Z}

\def\de{\delta}

\def\ZZ{\mathbb{Z}}

\renewcommand{\eqref}[1]{{\rm (\ref{#1})}}

\begin{document}

\title[Cluster multiplication theorem in the quantum cluster algebra of type $A_{2}^{(2)}$]
{Cluster multiplication theorem in the quantum cluster algebra of type $A_{2}^{(2)}$}

\author{Liqian Bai, Xueqing Chen, Ming Ding and Fan Xu}
\address{Department of Applied Mathematics, Northwestern Polytechnical University, Xi'an, Shaanxi 710072, P. R. China}
\email{bailiqian@nwpu.edu.cn (L.Bai)}
\address{Department of Mathematics,
 University of Wisconsin-Whitewater\\
800 W. Main Street, Whitewater, WI.53190. USA}
\email{chenx@uww.edu (X.Chen)}
\address{School of Mathematical Sciences and LPMC,
Nankai University, Tianjin, P. R. China}
\email{m-ding04@mails.tsinghua.edu.cn (M.Ding)}
\address{Department of Mathematical Sciences\\
Tsinghua University\\
Beijing 100084, P. R. China} \email{fanxu@mail.tsinghua.edu.cn(F.Xu)}


\thanks{Liqian Bai was supported by NPU (No. 3102017OQD033), Ming Ding was supported by NSF of China (No. 11771217) and Specialized Research Fund for the Doctoral Program of Higher Education (No. 20130031120004) and Fan Xu was supported by NSF of China (No. 11471177).}




\maketitle

\begin{abstract}
The objective of the present paper is to prove cluster multiplication theorem in the quantum cluster algebra of type $A_{2}^{(2)}$. As corollaries, we obtain
bar-invariant $\mathbb{Z}[q^{\pm\frac{1}{2}}]$-bases  established in \cite{cds}, and naturally deduce the positivity of the elements in these
bases. One bar-invariant basis as the triangular basis of this quantum cluster algebra is also explicitly described.
\end{abstract}


\section{Introduction}

The concept of cluster algebras was introduced by Fomin and Zelevinsky \cite{ca1}\cite{ca2} in order
to develop an algebraic framework for understanding total positivity and canonical bases in semisimple algebraic groups.
As a noncommutative analogue of cluster algebras, quantum cluster algebras were defined by Berenstein and Zelevinsky in ~\cite{BZ2005}.
Under the specialization $q=1,$ the quantum cluster algebras are degenerated to cluster algebras.

For the classical cluster algebras, Sherman and Zelevinsky \cite{SZ} firstly
gave the cluster multiplication formula in rank $2$ cluster algebras of finite and affine types.
For the general case,  the cluster categories were introduced in ~\cite{BMRRT} as the categorification of acyclic cluster algebras.
Cluster algebras have a close link to quiver representations via cluster categories.
The link is explicitly characterized by the Caldero-Chapoton map \cite{CC} and the
Caldero-Keller multiplication theorems \cite{CK}\cite{CK-2}. The
Caldero-Chapoton map associates the objects in the cluster
categories to some Laurent polynomials, in particular, sends indecomposable rigid
objects to cluster variables. The remarkable Caldero-Keller multiplication
theorems show the multiplication rules between images of objects
under the Caldero-Chapoton map.  For
simply laced Dynkin quivers, Caldero and Keller constructed a
cluster multiplication formula (of finite type) between  two
generalized cluster variables in \cite{CK}. On the
one hand, this multiplication is similar to the multiplication in a dual Hall algebra
and unifies homological and geometric properties of cluster
categories and combinatorial properties of cluster algebras. Since cluster algebras were introduced  in order to study
canonical bases, it is important to construct integral bases of
cluster algebras. In the cluster theory, the Caldero-Chapoton map and the Caldero-Keller cluster multiplication theorem open a new
way to construct cluster algebras from 2-Calabi-Yau categories and play a
very important role to obtain structural results such as bases with good properties,
positivity conjecture, denominator conjecture and so on ~\cite{CK}\cite{dx}\cite{DXX}.
The cluster multiplication formula of finite type was generalized to
affine type in \cite{Hubery2005} and to any type in \cite{XiaoXu}.  Palu \cite{Palu2} further extended the formula to
2-Calabi-Yau categories with cluster tilting objects.  In~\cite{DWZ}, the full generalization of the  Caldero-Chapoton map was obtained for quivers with potentials. Following this link, some good bases have been constructed for finite and affine cluster algebras \cite{CK}\cite{calzel}\cite{DXX}.


It is natural to ask the question: what are the quantum analogue of the link? Recently,
Rupel \cite{rupel} defined a quantum version of the
Caldero-Chapoton map 
for the quantum cluster algebras over finite fields associated with valued acyclic quivers
and he conjectured that cluster variables could be expressed as images of
indecomposable rigid objects under the quantum Caldero-Chapoton
formula. A key ingredient of the conjecture is to confirm the
mutation rules between quantum cluster variables given by
\cite{BZ2005}. Most recently, the conjecture has been proved by Qin
\cite{fanqin} for acyclic equally valued quivers. There the author
constructed a quantum cluster multiplication formula and then
confirmed the mutation rules between quantum cluster
variables. Note that Qin verified the formula for the usual quantum cluster algebras through
the existence of counting polynomials instead of working over the finite field.
The quantum Caldero-Chapoton maps are further generalized in ~\cite{DMSS}\cite{Davison}.
In ~\cite{dx-2}, Ding and Xu proved a multiplication theorem
for acyclic quantum cluster algebras which
generalized the quantum cluster multiplication formula in
\cite{fanqin} and could be viewed as a quantum analogue of the
one-dimensional Caldero-Keller multiplication theorem discussed in \cite{CK-2}.
Compared to the role which the Caldero-Keller multiplication
theorems play for cluster algebras, the quantum multiplication theorem is
worthy of highlighting and also reflects the information and the
difficulty to prove the more general quantum analog of the
Caldero-Keller multiplication theorems. By using this multiplication theorem, it is not too difficult to
construct some good $\mathbb{ZP}$-bases in quantum cluster
algebras of finite and affine types.  By specializing $q$ and
coefficients to $1$, these bases induce the good bases for cluster
algebras of finite \cite{CK} and affine types \cite{DXX},
respectively.

One may expect to explicitly express the multiplication of two basis elements in terms of basis elements in quantum cluster algebras, i.e., to get structure constants clearly and explicitly. Ding and Xu ~\cite{dx} gave the cluster multiplication formula in the quantum cluster algebra of the Kronecker quiver.
By using the multiplication formula, they constructed bar-invariant bases of the quantum cluster algebra of Kronecker quiver as quantum analogues of the canonical basis, semicanonical basis and dual semicanonical basis of the corresponding cluster algebra. As a byproduct, they also proved positivity of the elements in these bases. In this paper, we construct a nontrivial cluster multiplication formula in the quantum cluster algebra of the non-simply-laced valued quiver $A_{2}^{(2)}$, which is parallel to the results obtained for Kornecker quiver but is not a trivial generalization. This formula yields some important properties of quantum cluster algebra of type $A_{2}^{(2)}$. For example, we construct three integral bases of this quantum cluster algebra.  It is worthy of highlighting that the basis $\mathcal{B}$ we obtained coincide with ``quantum greedy basis" (or ``quantum atomic basis", ``quantum theta basis") defined in ~\cite{LLRZ}. In general, quantum greedy basis does not have positive structure constants. But it does in this case following from the main result  Theorem ~\ref{theorem2} of the paper . Whether quantum greedy bases exist or not for general cases are still not known.
The basis $\mathcal{S}$ is closely related to the dual canonical basis of a quantum unipotent cell (subset of the dual canonical basis of the $U_q(n))$, when the valued quiver is attached some correct frozen variables. In the last section, we prove that $\mathcal{S}$ is the triangular basis in the sense of Qin~\cite{fanqin2}, or the triangular basis in the sense of Berenstein-Zelevinsky ~\cite{BZ2014}. Note that Qin~\cite{fanqin3} proved that both definitions of the triangular bases are equivalent for the seeds associated with
acyclic quivers and for the seeds associated with bipartite skew-symmetrizable matrices. The basis $\mathcal{D}$ is similar to the quantum generic basis (quantum dual semi-canonical basis) in ~\cite{KQ}. To construct explicit multiplication formula is still open for general cases.

\section{Preliminaries}

For the terminology related to quantum cluster algebras, one can refer to ~\cite{BZ2005} for more details, to quantum cluster algebra of type $A_2^{(2)}$, refer to ~\cite{cds}.

In this paper, we consider the valued quiver associated to the compatible pair $(\Lambda,B)$ where
\begin{equation*}
\Lambda=\left(
  \begin{array}{cc}
    0 & 1 \\
    -1 & 0 \\
  \end{array}
\right)
\text{~and~}
B=\left(
  \begin{array}{cc}
    0 & 1 \\
    -4 & 0 \\
  \end{array}
\right).
\end{equation*}
Note that $\Lambda^{T}B=\left(
  \begin{array}{cc}
    4 & 0 \\
    0 & 1 \\
  \end{array}
\right)$.

Now let $q$ denote the formal variable and $\mathcal{F}$ be the skew field of fractions of the \emph{quantum torus} $\mathcal{T}=\ZZ[q^{\pm \frac{1}{2}}]\langle X_1,X_2~|~X_1X_2=qX_2X_1\rangle$. The \emph{quantum cluster algebra} $\A_q(1,4)$ is the $\ZZ[q^{\pm \frac{1}{2}}]$-subalgebra of $\mathcal{F}$ generated by the cluster variables $X_k$ for $k\in \mathbb{Z}$, recursively defined by
$$
X_{k-1}X_{k+1}=\left\{
  \begin{aligned}
    q^{\frac{1}{2}}X_k+1,& ~\text{ if }~k~\text{is odd};  \\
    q^{2}X_{k}^{4}+1,& ~\text{ if }~k~\text{is even}.
  \end{aligned}
  \right.
$$

Note that $X_k\in\mathcal{T}$ by the well-known quantum Laurent phenomenon ~\cite{BZ2005}.
For each $(a,b)\in\ZZ^{2}$, if we set $X^{(a,b)}=q^{-\frac{1}{2}ab}X_{1}^{a}X_{2}^{b}$, then
$X^{(a,b)}X^{(c,d)}=q^{-\frac{1}{2}(bc-ad)}X^{(a+c,b+d)}$. The  $\mathbb{Z}$-linear bar-involution on the based quantum torus $\mathcal{T}$ is defined as follows:
$$\overline{q^{\frac{r}{2}}X^{(a,b)}}=q^{-\frac{r}{2}}X^{(a,b)}, \ \ \ \text{for any}~r,a \text{ and } b\in \mathbb{Z}.$$

In \cite{cds}, the authors constructed three kinds of bar-invariant $\ZZ[q^{\pm \frac{1}{2}}]$-bases of the quantum cluster algebra $\A_q(1,4)$ by using the standard monomials discussed in \cite{BZ2005}.  We now briefly recall some notations and results in ~\cite{cds}.

We define that
\begin{align*}
 X_\delta:=&X^{(-1,-2)}+X^{(-1,2)}+X^{(1,-2)}+(q^{-\frac{1}{2}}+q^{\frac{1}{2}})X^{(0,-2)} &\\
 =&qX_{0}^{2}X_3-q^{2}(qX_1+q^{-\frac{1}{2}}+q^{\frac{1}{2}})X_{2}^{2}.&
\end{align*}

Let
$$
 \begin{aligned}
&\mathcal{B}=\{q^{-\frac{1}{2}ab}X^{a}_{m}X^{b}_{m+1}~|~m\in\ZZ,(a,b)\in\ZZ^{2}_{\geq0}\}\cup \{F_{n}(X_\delta)\},\\
&\mathcal{S}=\{q^{-\frac{1}{2}ab}X^{a}_{m}X^{b}_{m+1}~|~m\in\ZZ,(a,b)\in\ZZ^{2}_{\geq0}\}\cup \{S_{n}(X_\delta)\},\\
&\mathcal{D}=\{q^{-\frac{1}{2}ab}X^{a}_{m}X^{b}_{m+1}~|~m\in\ZZ,(a,b)\in\ZZ^{2}_{\geq0}\}\cup \{X_{\delta}^{n}\},
  \end{aligned}
$$
where $F_{n}(x)$ and $S_{n}(x)$ are well-known Chebyshev polynomials defined by
  $$
  F_0(x)=1,F_1(x)=x, F_2(x)=x^2-2, F_{n+1}(x)=F_{n}(x)x-F_{n-1}(x)~\text{for}~n\geq2,
  $$
  $$
  S_0(x)=1,S_1(x)=x, S_2(x)=x^2-1, S_{n+1}(x)=S_{n}(x)x-S_{n-1}(x)~\text{for}~n\geq2,
  $$
and $F_n(x)=S_n(x)=0$ for $n<0$.
The homomorphism $\sigma_2:\A_{q}(1,4) \rightarrow \A_{q}(1,4)$ defined by $X_m\mapsto X_{m+2}$ and $q^{\pm\frac{1}{2}}\mapsto q^{\pm\frac{1}{2}}$ is an automorphism of $\A_q(1,4)$~\cite[Section 4]{cds}. Note that $\sigma_2 (X_\delta)=X_\delta$.

We can define a partial order $\leq$ on $\mathbb{Z}^2$ as follows:
$(r_1,r_2)\leq(s_1,s_2)$ if $r_1\leq s_1$ and $r_2\leq s_2$ for $(r_1,r_2), (s_1,s_2)\in \mathbb{Z}^2$. Moreover if there exists some $i \in \{1,2\}$ such that $r_i< s_i$, we will write $(r_1,r_2)<(s_1,s_2)$. In ~\cite{cds}, the authors showed that every element in $\{X_n$ $(n\in \mathbb{Z}\setminus \{1,2\}), F_{n}(X_\delta)(n\geq 1), S_{n}(X_\delta)(n\geq 1)\}$ has a minimal non-zero term $X^{(a,b)}$ according to the partial order $\leq$. The vector $(-a,-b)$ associated to this minimal non-zero term $X^{(a,b)}$ of the corresponding element will be called the denominator vector.  Then by using the standard monomials, they proved that $\mathcal{B}$, $\mathcal{S}$ and $\mathcal{D}$ are bar-invariant $\ZZ[q^{\pm \frac{1}{2}}]$-bases of the quantum cluster algebra $\A_q(1,4)$. Unfortunately, the structure constants and the positivity are not presented in this construction. This motivated our study for the multiplication formulas.

\section{Cluster multiplication theorem and positive bases}

In this section, we mainly prove the cluster multiplication theorem of the quantum cluster algebra $\A_q(1,4)$. Since the element $X_\delta$ stated in previous section plays a crucial importance in the cluster multiplication theorem, we firstly address another expression of this element.

\begin{lemma}\label{lem1}
In $\A_q(1,4)$, we have that
$
X_\delta=q^{-1}X_{4}^{2}X_1-q^{-2}(q^{-1}X_3+q^{-\frac{1}{2}}+q^{\frac{1}{2}})X_{2}^{2}.
$
\end{lemma}
\begin{proof}
Note that $X_3=X^{(-1,4)}+X^{(-1,0)}$ and $X_4=X^{(-1,3)}+X^{(-1,-1)}+X^{(0,-1)}$, then we have that
\begin{align*}
&q^{-1}X_{4}^{2}X_1-q^{-2}(q^{-1}X_3+q^{-\frac{1}{2}}+q^{\frac{1}{2}})X_{2}^{2}\\
=& q^{-1}(X^{(-1,3)}+X^{(-1,-1)}+X^{(0,-1)})^{2}X^{(1,0)} -(q^{-3}(X_3+q^{\frac{1}{2}}+q^{\frac{3}{2}})X_{2}^{2}) \\
=&q^{-4}X^{(-1,6)}+X^{(-1,-2)}+X^{(1,-2)}+(q^{-\frac{1}{2}}+q^{\frac{1}{2}})q^{-2}X^{(0,2)}\\
 &+(q^{-2}+q^{2})q^{-2}X^{(-1,2)} +(q^{-\frac{1}{2}}+q^{\frac{1}{2}})X^{(0,-2)}\\
 &-(q^{-4}X^{(-1,6)}+q^{-4}X^{(-1,2)}+(q^{-\frac{5}{2}}+q^{-\frac{3}{2}})X^{(0,2)})\\
=&X^{(-1,-2)}+X^{(1,-2)}+X^{(-1,2)}+ (q^{-\frac{1}{2}}+q^{\frac{1}{2}})X^{(0,-2)}= X_\delta.
\end{align*}

\end{proof}

The following proposition is a special case discussed in \cite{BCDX}, here we give an alternative proof by using the above lemma.
\begin{proposition}\label{prop-gene}
The quantum cluster algebra $\A_{q}(1,4)$ is the $\ZZ[q^{\pm\frac{1}{2}}]$-algebra generated by $\{X_m,X_{m+1},X_{m+2},X_{m+3}\}$ for any $m\in\ZZ$.
\end{proposition}
\begin{proof}
By the definition of $X_\delta$, we know that $X_\delta\in\ZZ[q^{\pm\frac{1}{2}}]\langle X_0,X_1,X_2,X_3\rangle$. We have that $X_\delta\in\ZZ[q^{\pm\frac{1}{2}}]\langle X_1,X_2,X_3,X_4\rangle$ by Lemma \ref{lem1}. Then through the automorphism $\sigma_2$, we can deduce that $X_\delta\in\ZZ[q^{\pm\frac{1}{2}}]\langle X_m,X_{m+1},X_{m+2},X_{m+3}\rangle$ for any $m\in\ZZ$.

Note that for any $n\in\ZZ$, we have that $X_{2n}X_\delta=q^{-\frac{1}{2}}X_{2n-2}+q^{\frac{1}{2}}X_{2n+2}$ (see \cite[Proposition 4.2]{cds}). Then we can deduce that $X_{2n}\in\ZZ[q^{\pm\frac{1}{2}}]\langle X_m,X_{m+1},X_{m+2},X_{m+3}\rangle$ for any $n\in\ZZ$. Since $X_{2n-2}X_{2n}=q^{\frac{1}{2}}X_{2n-1}+1$, we obtain that all  cluster variables belong to $\ZZ[q^{\pm\frac{1}{2}}]\langle X_m,X_{m+1},X_{m+2},X_{m+3}\rangle$. Thus
$
\A_q(1,4)= \ZZ[q^{\pm\frac{1}{2}}]\langle X_m,X_{m+1},X_{m+2},X_{m+3}\rangle.
$
\end{proof}

For each $n\in\ZZ$, we denote by
\begin{align*}
\langle n\rangle=\left\{
\begin{aligned}
1,&~\text{if~}n~\text{is~odd};\\
2,&~\text{if~}n~\text{is~even}.
\end{aligned}
\right.
\end{align*}

Let $x\in\mathbb{R}$, we have the floor function $\lfloor x\rfloor:=\text{max}\{m\in\ZZ~|~m\leq x\}$ and the ceiling function $\lceil x\rceil:=\text{min}\{m\in\ZZ~|~m\geq x\}$.

For any $m>n\geq1$, it is easy to show that (see \cite[Proposition 4.2]{cds}):
$$
F_n(X_\delta)F_m(X_\delta)=F_{m+n}(X_\delta)+F_{m-n}(X_\delta)~~
\text{ and }
~~F_n(X_\delta)F_n(X_\delta)=F_{2n}(X_\delta)+2.
$$
The following cluster  multiplication theorem is the main result of the present paper.
\begin{theorem}\label{theorem2}
In $\A_q(1,4)$, we have that
\begin{enumerate}
\item \begin{enumerate}
\item[(i)]if $m$ is even and $n$ is positive, then
\begin{equation}\label{equation1}
X_{m}F_n(X_{\delta})=q^{-\frac{n}{2}}X_{m-2n}+q^{\frac{n}{2}}X_{m+2n}.
\end{equation}
\item[(ii)]if $m$ is odd and $n$ is positive, then
\begin{align}\label{equation2}
& X_mF_n(X_\delta)\\=&q^{-n}X_{m-n}^{\langle m-n\rangle}+q^{n}X_{m+n}^{\langle m+n\rangle}\nonumber
+\sum\limits_{k\geq1} (\sum\limits_{l=1}^{k} (q^{-\frac{4l-1}{2}}+q^{-\frac{4l-3}{2}}+q^{\frac{4l-3}{2}}+q^{\frac{4l-1}{2}}))F_{n-2k}(X_{\de}).
\end{align}
\end{enumerate}
\item if $m$ is even and $n$ is positive, then
\begin{equation}\label{equation3}
X_mX_{m+2n}=q^{\frac{n}{2}}X_{m+n}^{\langle m+n\rangle}+\sum\limits_{k\geq1}\big(\sum\limits_{l=1}^{2k-1}q^{-\frac{n+1}{2}+l}\big)F_{n-2k+1}(X_{\delta}).
\end{equation}
\item if $m$ is even and $n$ is positive odd, then
\begin{align}\label{equation4}
\nonumber
&X_{m-n}X_{m}\\=&\sum\limits_{1<2k<n}\big(\sum\limits_{l=1}^{\text{min}(4k,n-2k)}q^{-\frac{1}{2}-k+l}\big)X_{m-4k}
+\left\{
\begin{aligned}
q^{\frac{n}{2}}X^{3}_{m-\frac{2}{3}n},{\hskip 1.9cm}~~~~~~~~~ &~~ n\equiv 0~(\rm{mod}~3) \\
q^{\frac{n-1}{2}}X_{\lfloor m-\frac{2}{3}n\rfloor}X_{\lceil m-\frac{2}{3}n\rceil},&~~ \rm{otherwise,}
\end{aligned}
\right.
\end{align}
and
\begin{align}\label{equation5}
\nonumber
& X_{m+n}X_{m}\\=&\sum\limits_{1<2k<n}\big(\sum\limits_{l=1}^{\text{min}(4k,n-2k)}q^{\frac{1}{2}+k-l}\big)X_{m+4k}
+\left\{
\begin{aligned}
q^{-\frac{n}{2}}X^{3}_{m+\frac{2}{3}n},{\hskip 2.0cm}& ~n\equiv 0~(\rm{mod}~3)~~~~~~~ \\
q^{-\frac{n+1}{2}}X_{\lfloor m+\frac{2}{3}n\rfloor}X_{\lceil m+\frac{2}{3}n\rceil},&~\rm{otherwise.}
\end{aligned}
\right.
\end{align}
\item
if $m$ is odd and $n$ is positive, then
\begin{align}\label{equation6}
\nonumber
& X_mX_{m+2n}\\=&q^{2n}X_{m+n}^{2\langle m+n\rangle}+\sum\limits_{k=1}^{n-1}(\sum\limits\limits_{l=1}^{4\text{min}(k,n-k)}q^{-\frac{1}{2}+l}) X_{m+2n-2k}
+\sum\limits_{k\geq1}c_{n,k}F_{2n-2k}(X_{\delta}),
\end{align}
where
\begin{align*}
c_{n,k}=\sum\limits_{i=1}^{k}a_i(q^{-2(n-i)-1}+q^{4k-2(n+i)+1})+ \sum\limits_{i=1}^{k-1}b_i(q^{-2(n-i)}+q^{4k-2(n+i)}) +b_kq^{-2(n-k)}
\end{align*}
and $a_j=\frac{j(j-1)}{2}$, $b_j=\frac{j(j-1)}{2}+\lceil \frac{j}{2}\rceil$ for positive integer $j$.
\end{enumerate}
\end{theorem}

\begin{proof}
(1) In order to prove (\ref{equation1}), it suffices to show that
$$
X_2F_n(X_\delta)=q^{-\frac{n}{2}}X_{2-2n}+q^{\frac{n}{2}}X_{2+2n}.
$$
We will prove the claim by induction on $n$.  When $n=1$, it follows from \cite[Proposition 4.2]{cds}. When $n=2$, we have that
\begin{align*}
&X_2F_2(X_\delta)=X_2(X_{\delta}^2-2)=q^{-\frac{1}{2}}X_0X_\delta+q^{\frac{1}{2}}X_4X_\delta-2X_2\\
=&(q^{-1}X_{-2}+X_2)+(X_2+qX_6)-2X_2 =q^{-1}X_{-2}+qX_6.
\end{align*}
Assume that $X_2F_n(X_\delta)=q^{-\frac{n}{2}}X_{2-2n}+q^{\frac{n}{2}}X_{2+2n}$ for $n\geq2$. Then
\begin{align*}
&X_2F_{n+1}(X_\delta)=X_2F_n(X_\delta)X_\delta-X_2F_{n-1}(X_\delta)\\
=&q^{-\frac{n+1}{2}}X_{-2n}+q^{-\frac{n-1}{2}}X_{4-2n}+q^{\frac{n-1}{2}}X_{2n}+q^{\frac{n+1}{2}}X_{4+2n} -q^{-\frac{n-1}{2}}X_{4-2n}-q^{\frac{n-1}{2}}X_{2n}\\
=&q^{-\frac{n+1}{2}}X_{-2n}+q^{\frac{n+1}{2}}X_{4+2n}.
\end{align*}
To prove (\ref{equation2}), it suffices to show that
\begin{align}\label{equ2}
&X_1F_n(X_\delta)\\
=&q^{-n}X_{1-n}^{\langle 1-n\rangle}+q^{n}X_{1+n}^{\langle 1+n\rangle}\nonumber+\sum\limits_{k\geq1} \big(\sum\limits_{l=1}^{k} (q^{-\frac{4l-1}{2}}+q^{-\frac{4l-3}{2}}+q^{\frac{4l-3}{2}}+q^{\frac{4l-1}{2}})\big)F_{n-2k}(X_{\de}).
\end{align}
When $n=1$, we have that $X_1X_\delta=q^{-1}X_0^2+qX_2^2$ by \cite[Proposition 4.2]{cds}. When $n=2$, we have that
\begin{align*}
&X_1F_2(X_\delta)=X_1(X_\delta^2-2)=(q^{-1}X_0^2+qX_{2}^{2})X_\delta-2X_1\\
\pagebreak
=&q^{-\frac{3}{2}}(q^{-\frac{1}{2}}X_{-1}+1)+q^{-\frac{1}{2}}(q^{\frac{1}{2}}X_{1}+1)+ q^{\frac{1}{2}}(q^{-\frac{1}{2}}X_{1}+1)+q^{\frac{3}{2}}(q^{\frac{1}{2}}X_{3}+1)-2X_1\\
=&q^{-2}X_{-1}+q^{2}X_3+(q^{-\frac{3}{2}}+q^{-\frac{1}{2}}+q^{\frac{1}{2}}+q^{\frac{3}{2}}).
\end{align*}
Assume that (\ref{equ2}) is true. Note that $X_1F_{n+1}(X_\delta)=X_1F_n(X_\delta)X_\delta-X_1F_{n-1}(X_\delta)$. When $n$ is even,
\begin{align*}
&X_1F_{n+1}(X_\delta)\\
=&[q^{-n}X_{1-n}+q^{n}X_{1+n}+\sum\limits_{k\geq1}(\sum\limits_{l=1}^{k}(q^{-\frac{4l-1}{2}}+q^{-\frac{4l-3}{2}} +q^{\frac{4l-3}{2}}+q^{\frac{4l-1}{2}})F_{n-2k}(X_\delta))]X_\delta\\
&-q^{1-n}X^{2}_{2-n}-q^{n-1}X_{n}^{2}-\sum\limits_{k\geq1}(\sum\limits_{l=1}^{k}(q^{-\frac{4l-1}{2}}+q^{-\frac{4l-3}{2}} +q^{\frac{4l-3}{2}}+q^{\frac{4l-1}{2}})F_{n-1-2k}(X_\delta))\\
=&q^{-n-1}X_{-n}^{2}+q^{n+1}X^{2}_{n+2}+\sum\limits_{k=1}^{\frac{n}{2}-1} \sum\limits_{l=1}^{k}(q^{-\frac{4l-1}{2}}+q^{-\frac{4l-3}{2}} +q^{\frac{4l-3}{2}}+q^{\frac{4l-1}{2}})F_{n+1-2k}(X_\delta)\\
&+\sum\limits_{l=1}^{\frac{n}{2}}(q^{-\frac{4l-1}{2}}+q^{-\frac{4l-3}{2}} +q^{\frac{4l-3}{2}}+q^{\frac{4l-1}{2}})X_\delta \\
=&q^{-n-1}X^{2}_{-n}+q^{n+1}X_{n+2}^{2}+\sum\limits_{k=1}^{\frac{n}{2}}\sum\limits_{l=1}^{k} (q^{-\frac{4l-1}{2}}+q^{-\frac{4l-3}{2}} +q^{\frac{4l-3}{2}}+q^{\frac{4l-1}{2}})F_{n+1-2k}(X_\delta).
\end{align*}

The proof for the odd $n$ is similar.

(2) For $n\geq0$, it suffices to show that
\begin{equation}\label{equ3}
X_2X_{2+2n}=q^{\frac{n}{2}}X_{2+n}^{<2+n>}+\sum\limits_{k\geq1}(\sum\limits_{l=1}^{2k-1}q^{-\frac{n+1}{2}+l}) F_{n-2k+1}(X_\delta).
\end{equation}

When $n=1$, it is the exchange relation. When $n=2$, we have that
\begin{align*}
X_2X_6
=&q^{-\frac{1}{2}}X_2X_4X_\delta-q^{-1}X_{2}^{2}=q^{-\frac{1}{2}}(q^{\frac{1}{2}}X_3+1)X_\delta-q^{-1}X_{2}^{2}\\
=&X_3X_\delta+q^{-\frac{1}{2}}X_\delta-q^{-1}X_{2}^{2}=qX_{4}^{2}+q^{-\frac{1}{2}}X_\delta.
\end{align*}
Assume that (\ref{equ3}) is true. Now we calculate $X_2X_{4+2n}$. Note that
$$
X_{2+2n}X_\delta=q^{-\frac{1}{2}}X_{2n}+q^{\frac{1}{2}}X_{2n+4},$$
we have $X_{2n+4}=q^{-\frac{1}{2}}X_{2+2n}X_\delta-q^{-1}X_{2n}$.

When $n$ is even, it follows that
\begin{align*}
&X_2X_{4+2n}
=q^{-\frac{1}{2}}X_{2}X_{2+n}X_\delta-q^{-1}X_2X_{2n}\\
=&q^{\frac{n-1}{2}}X^{2}_{2+n}X_\delta+\sum\limits_{k\geq1}(\sum\limits_{l=1}^{2k-1}q^{-\frac{n}{2}-1+l}) F_{n+1-2k}(X_\delta)X_\delta -q^{\frac{n-3}{2}}X_{n+1}\\&-\sum\limits_{k\geq1}(\sum\limits_{l=1}^{2k-1}q^{-\frac{n}{2}-1+l}) F_{n-2k}(X_\delta).
\end{align*}

Note that $X_{2+n}^{2}X_\delta=q^{-1}X_{n+1}+qX_{n+3}+(q^{\frac{1}{2}}+q^{-\frac{1}{2}})$.

Therefore, we have that
\begin{align*}
&X_2X_{4+2n} \\
=&q^{\frac{n-3}{2}}X_{n+1}+q^{\frac{n+1}{2}}X_{n+3}+q^{\frac{n-1}{2}}(q^{-\frac{1}{2}}+q^{\frac{1}{2}}) +\sum\limits_{k\geq1}(\sum\limits_{l=1}^{2k-1}q^{-\frac{n}{2}-1+l})F_{n+1-2k}(X_\delta)X_\delta\\
&-q^{\frac{n-3}{2}}X_{n+1} -\sum\limits_{k\geq1}(\sum\limits_{l=1}^{2k-1}q^{-\frac{n}{2}-1+l})F_{n-2k}(X_\delta)\\
=&q^{\frac{n+1}{2}}X_{n+3}+(q^{\frac{n}{2}-1}+q^{\frac{n}{2}})+ \sum\limits_{k=1}^{\frac{n}{2}-1}(\sum\limits_{l=1}^{2k-1}q^{-\frac{n}{2}-1+l})F_{n+2-2k}(X_\delta) +\sum\limits_{l=1}^{n-1}q^{-\frac{n}{2}-1+l}X_{\delta}^{2}\\
&-\sum\limits_{l=1}^{n-1}q^{-\frac{n}{2}-1+l}\\
=&q^{\frac{n+1}{2}}X_{n+3}+ \sum\limits_{k=1}^{\frac{n}{2}+1} (\sum\limits_{l=1}^{2k-1}q^{-\frac{n}{2}-1+l})F_{n+2-2k}(X_\delta).
\end{align*}

Similarly, we can prove the statement for odd $n$.

(3) To prove (\ref{equation4}), it suffices to show that for a positive odd integer $n$, we have that
\begin{align}\label{equ4}
\nonumber
&X_{1}X_{1+n}\\
=&\sum\limits_{1<2k<n}\big(\sum\limits_{l=1}^{\text{min}(4k,n-2k)}q^{-\frac{1}{2}-k+l}\big)X_{n+1-4k}
+\left\{
\begin{aligned}
q^{\frac{n}{2}}X^{3}_{1+\frac{n}{3}},{\hskip 1.75cm}&~n\equiv 0~(\rm{mod}~3);\\
q^{\frac{n-1}{2}}X_{\lfloor 1+\frac{n}{3}\rfloor}X_{\lceil 1+\frac{n}{3}\rceil},&~\rm{otherwise.}
\end{aligned}
\right.
\end{align}

When $n=1$, it is trivial. When $n=3$, note that $X_4=q^{-\frac{1}{2}}X_2X_\delta-q^{-1}X_0$ by (\ref{equation1}). It follows that
\begin{align*}
&X_1X_4=q^{-\frac{1}{2}}X_1X_2X_\delta-q^{-1}X_1X_0= q^{\frac{1}{2}}X_2X_1X_\delta-q^{-1}X_1X_0\\
=&q^{\frac{1}{2}}X_2(q^{-1}X^{2}_{0}+q^{}X^{2}_{2})-q^{-1}X_1X_0
=q^{-\frac{1}{2}}X_0+q^{\frac{3}{2}}X_{2}^{3}.
\end{align*}
Since $X_4X_\delta=q^{-\frac{1}{2}}X_2+q^{\frac{1}{2}}X_6$ and $X_6=q^{-\frac{1}{2}}X_4X_\delta-q^{-1}X_2$, when $n=5$, we have that
\begin{align*}
&X_1X_6=X_1(q^{-\frac{1}{2}}X_4X_\delta-q^{-1}X_2)=q^{-1}X_0X_\delta+qX_{2}^{3}X_\delta-q^{-1}X_1X_2\\
=&q^{-\frac{3}{2}}X_{-2}+q^{-\frac{1}{2}}X_2+q^{\frac{1}{2}}X_{2}^{2}X_0+q^{\frac{3}{2}}X_{2}^{2}X_4-q^{-1}X_1X_2\\
=&q^{-\frac{3}{2}}X_{-2}+(q^{-\frac{1}{2}}+q^{\frac{1}{2}}+q^{\frac{3}{2}})X_2+q^2X_2X_3.
\end{align*}

Assume that (\ref{equ4}) is true. Note that $X_1X_{3+n}=q^{-\frac{1}{2}}X_1X_{1+n}X_\delta-q^{-1}X_1X_{n-1}$.

If $n\equiv0~(\rm{mod}~3)$, then
$$
X_1X_{n+1}=\sum\limits_{1<2k<n}(\sum\limits_{l=1}^{\text{min}(4k,n-4k)}q^{-\frac{1}{2}-k+l})X_{n+1-4k} +q^{\frac{n}{2}}X^{3}_{1+\frac{n}{3}}
$$
and
$$
q^{-1}X_1X_{n-1}=\sum\limits_{1<2k<n-2} (\sum\limits_{l=1}^{\text{min}(4k,n-2-2k)}q^{-\frac{3}{2}-k+l})X_{n-1-4k} +q^{\frac{n-5}{2}}X_{\frac{n}{3}}X_{1+\frac{n}{3}}.
$$
We then get
\begin{align*}
&q^{-\frac{1}{2}}X_1X_{n+1}X_\delta\\
=&\sum\limits_{1<2k<n} (\sum\limits_{l=1}^{\text{min}(4k,n-2k)}q^{-1-k+l})(q^{-\frac{1}{2}}X_{n-1-4k} +q^{\frac{1}{2}}X_{n+3-4k})\\ &+q^{\frac{n-1}{2}}X^{2}_{1+\frac{n}{3}}(q^{-\frac{1}{2}}X_{\frac{n}{3}-1} +q^{\frac{1}{2}}X_{3+\frac{n}{3}})\\
=&\sum\limits_{1<2k<n} (\sum\limits_{l=1}^{\text{min}(4k,n-2k)}q^{-\frac{3}{2}-k+l})X_{n-1-4k}
+\sum\limits_{1<2k<n} (\sum\limits_{l=1}^{\text{min}(4k,n-2k)}q^{-\frac{1}{2}-k+l})X_{n+3-4k}\\
&+q^{\frac{n-3}{2}}X_{1+\frac{n}{3}}X_{\frac{n}{3}}+q^{\frac{n+1}{2}}X_{1+\frac{n}{3}}X_{2+\frac{n}{3}} +(q^{\frac{n-2}{2}}+q^{\frac{n}{2}})X_{1+\frac{n}{3}}.
\end{align*}

Note that $\lfloor\frac{n-2}{6}\rfloor=\frac{n-3}{6}$, $\lceil\frac{n-2}{6}\rceil=\frac{n+3}{6}$, $\lfloor\frac{n}{6}\rfloor=\frac{n-3}{6}$, $\lceil\frac{n}{6}\rceil=\frac{n+3}{6}$, $\lfloor\frac{n+2}{6}\rfloor=\frac{n-3}{6}$ and $\lceil\frac{n+2}{6}\rceil=\frac{n+3}{6}$ since $n\equiv0$ (mod $3$). Then we have
\begin{align}\label{reln0}
\left\{
\begin{aligned}
4k<n-2-2k,&~\rm{if~}1\leq k\leq \frac{n-3}{6},\\
4k>n-2-2k,&~\rm{if~}\frac{n+3}{6}\leq k\leq \frac{n-3}{2},\\
4k<n-2k,{\hskip 0.6cm}&~\rm{if~}1\leq k\leq \frac{n-3}{6},\\
4k>n-2k,{\hskip 0.6cm}&~\rm{if~}\frac{n+3}{6}\leq k\leq \frac{n-1}{2},\\
4k<n+2-2k,&~\rm{if~}1\leq k\leq \frac{n-3}{6},\\
4k>n+2-2k,&~\rm{if~}\frac{n+3}{6}\leq k\leq \frac{n+1}{2}.
\end{aligned}
\right.
\end{align}

It follows that
\begin{align*}
&\sum\limits_{1<2k<n}(\sum\limits_{l=1}^{\text{min}(4k,n-2k)}q^{-\frac{3}{2}-k+l})X_{n-1-4k} -\sum\limits_{1<2k<n-2}(\sum\limits_{l=1}^{\text{min}(4k,n-2-2k)}q^{-\frac{3}{2}-k+l})X_{n-1-4k}\\
=&\sum\limits_{k=\frac{n-3}{6}}^{\frac{n-3}{2}}(\sum\limits_{l=n-1-2k}^{n-2k}q^{-\frac{3}{2}-k+l})X_{n-1-4k} +q^{-\frac{n}{2}}X_{1-n}-(q^{\frac{n-2}{2}}+q^{\frac{n}{2}})X_{1+\frac{n}{3}}.
\end{align*}
Hence
\begin{align*}
X_1X_{n+3}
=&\sum\limits_{k=\frac{n+3}{6}}^{\frac{n-1}{2}}(\sum\limits_{l=n+1-2k}^{n+2-2k}q^{-\frac{1}{2}-k+l})X_{n+3-4k} +q^{-\frac{n}{2}}X_{1-n}\\
&+\sum\limits_{1<2k<n}(\sum\limits_{l=1}^{\text{min}(4k,n-2k)}q^{-\frac{1}{2}-k+l})X_{n+3-4k} +q^{\frac{n+1}{2}}X_{1+\frac{n}{3}}X_{2+\frac{n}{3}}\\
=&\sum\limits_{1<2k<n+2}(\sum\limits_{l=1}^{\text{min}(4k,n+2-2k)}q^{-\frac{1}{2}-k+l})X_{n+3-4k} +q^{\frac{n+1}{2}}X_{1+\frac{n}{3}}X_{2+\frac{n}{3}}.
\end{align*}

If $n\equiv1~(\rm{mod}~3)$, then
$$
X_1X_{n+1}=\sum\limits_{1<2k<n}(\sum\limits_{l=1}^{\text{min}(4k,n-2k)}q^{-\frac{1}{2}-k+l})X_{n+1-4k}+ q^{\frac{n-1}{2}}X_{\frac{n+2}{3}}X_{\frac{n+5}{3}},
$$

$$
q^{-1}X_1X_{n-1}=\sum\limits_{1<2k<n-2}(\sum\limits_{l=1}^{\text{min}(4k,n-2-2k)}q^{-\frac{3}{2}-k+l})X_{n-1-4k} +q^{\frac{n-5}{2}}X_{\frac{n-1}{3}}X_{\frac{n+2}{3}}.
$$
It follows that
\begin{align*}
&q^{-\frac{1}{2}}X_1X_{n+1}X_{\delta}\\
=&\sum\limits_{1<2k<n}(\sum\limits_{l=1}^{\text{min}(4k,n-2k)}q^{-\frac{3}{2}-k+l})X_{n-1-4k} +\sum\limits_{1<2k<n}(\sum\limits_{l=1}^{\text{min}(4k,n-2k)}q^{-\frac{1}{2}-k+l})X_{n+3-4k}\\
&+q^{\frac{n-3}{2}}X_{\frac{n+2}{3}}X_{\frac{n-1}{3}}+q^{\frac{n-2}{2}}X_{\frac{n-1}{3}} +q^{\frac{n+2}{2}}X_{\frac{n+5}{3}}^{3}.
\end{align*}

Note that $\lfloor\frac{n-2}{6}\rfloor=\frac{n-7}{6}$, $\lceil\frac{n-2}{6}\rceil=\frac{n-1}{6}$, $\lfloor\frac{n}{6}\rfloor=\frac{n-1}{6}$, $\lceil\frac{n}{6}\rceil=\frac{n+5}{6}$, $\lfloor\frac{n+2}{6}\rfloor=\frac{n-1}{6}$ and $\lceil\frac{n+2}{6}\rceil=\frac{n+5}{6}$, therefore
\begin{align}\label{reln1}
\left\{
\begin{aligned}
4k<n-2-2k,&~\text{if~}1\leq k\leq \frac{n-7}{6},\\
4k>n-2-2k,&~\text{if~}\frac{n-1}{6}\leq k\leq \frac{n-3}{2},\\
4k<n-2k,{\hskip 0.6cm}&~\text{if~}1\leq k\leq \frac{n-1}{6},\\
4k>n-2k,{\hskip 0.6cm}&~\text{if~}\frac{n+5}{6}\leq k\leq \frac{n-1}{2},\\
4k<n+2-2k,&~\text{if~}1\leq k\leq \frac{n-1}{6},\\
4k>n+2-2k,&~\text{if~}\frac{n+5}{6}\leq k\leq \frac{n+1}{2}.
\end{aligned}
\right.
\end{align}

It follows that
\begin{align*}
&\sum\limits_{1<2k<n}(\sum\limits_{l=1}^{\text{min}(4k,n-2k)}q^{-\frac{3}{2}-k+l})X_{n-1-4k} -\sum\limits_{1<2k<n-2}(\sum\limits_{l=1}^{\text{min}(4k,n-2-2k)}q^{-\frac{3}{2}-k+l})X_{n-1-4k}\\
=&\sum\limits_{k=\frac{n+5}{6}}^{\frac{n-3}{2}}(\sum\limits_{l=n-1-2k}^{n-2k}q^{-\frac{3}{2}-k+l})X_{n-1-4k} +q^{\frac{n}{2}-2}X_{\frac{n-1}{3}}+q^{-\frac{n}{2}}X_{1-n}\\
=&\sum\limits_{k=\frac{n-1}{6}}^{\frac{n-3}{2}}(\sum\limits_{l=n-1-2k}^{n-2k}q^{-\frac{3}{2}-k+l})X_{n-1-4k} +q^{-\frac{n}{2}}X_{1-n}-q^{\frac{n}{2}-1}X_{\frac{n-1}{3}}.
\end{align*}
Hence
\begin{align*}
X_1X_{n+3}
=&\sum\limits_{k=\frac{n+5}{6}}^{\frac{n-1}{2}}(\sum\limits_{l=n+1-2k}^{n+2-2k}q^{-\frac{1}{2}-k+l})X_{n+3-4k} +q^{-\frac{n}{2}}X_{1-n}\\
&+\sum\limits_{1<2k<n}(\sum\limits_{l=1}^{\text{min}(4k,n-2k)}q^{-\frac{1}{2}-k+l})X_{n+3-4k} +q^{\frac{n+2}{2}}X_{\frac{n+5}{3}}^{3}\\
=&\sum\limits_{1<2k<n}(\sum\limits_{l=1}^{\text{min}(4k,n+2-2k)}q^{-\frac{1}{2}-k+l})X_{n+3-4k} +q^{-\frac{n}{2}}X_{1-n} +q^{\frac{n+2}{2}}X_{\frac{n+5}{3}}^{3}\\
=&\sum\limits_{1<2k<n+2}(\sum\limits_{l=1}^{\text{min}(4k,n+2-2k)}q^{-\frac{1}{2}-k+l})X_{n+3-4k} +q^{\frac{n+2}{2}}X_{\frac{n+5}{3}}^{3}.
\end{align*}

If $n\equiv2~(\rm{mod}~3)$, then
\begin{align*}
X_1X_{n+1}=\sum\limits_{1<2k<n}(\sum\limits_{l=1}^{\text{min}(4k,n-2k)}q^{-\frac{1}{2}-k+l})X_{n+1-4k}+q^{\frac{n-1}{2}} X_{\frac{n+1}{3}}X_{\frac{n+4}{3}},
\end{align*}
and
\begin{align*}
q^{-1}X_1X_{n-1}=\sum\limits_{1<2k<n-2}(\sum\limits_{l=1}^{\text{min}(4k,n-2-2k)}q^{-\frac{3}{2}-k+l})X_{n-1-4k} +q^{\frac{n-4}{2}} X^{3}_{\frac{n+1}{3}}.
\end{align*}

Note that
\begin{align*}
&q^{-\frac{1}{2}}X_1X_{1+n}X_\delta\\
=&\sum\limits_{1<2k<n}(\sum\limits_{l=1}^{\text{min}(4k,n-2k)}q^{-\frac{3}{2}-k+l})X_{n-1-4k} +\sum\limits_{1<2k<n}(\sum\limits_{l=1}^{\text{min}(4k,n-2k)}q^{-\frac{1}{2}-k+l})X_{n+3-4k}\\
&+q^{\frac{n-4}{2}}X_{\frac{n+1}{3}}^{3}+q^{\frac{n+1}{2}}X_{\frac{n+4}{3}}X_{\frac{n+7}{3}} +q^{\frac{n}{2}}X_{\frac{n+7}{3}}.
\end{align*}

Since $\lfloor\frac{n-2}{6}\rfloor=\frac{n-5}{6}$, $\lceil\frac{n-2}{6}\rceil=\frac{n+1}{6}$, $\lfloor\frac{n}{6}\rfloor=\frac{n-5}{6}$, $\lceil\frac{n}{6}\rceil=\frac{n+1}{6}$, $\lfloor\frac{n+2}{6}\rfloor=\frac{n+1}{6}$, $\lceil\frac{n+2}{6}\rceil=\frac{n+7}{6}$, we have that
\begin{align}\label{reln2}
\left\{
\begin{aligned}
4k<n-2-2k,&~\text{if~}1\leq k\leq \frac{n-5}{6},\\
4k>n-2-2k,&~\text{if~}\frac{n+1}{6}\leq k\leq \frac{n-3}{2},\\
4k<n-2k,{\hskip 0.6cm}&~\text{if~}1\leq k\leq \frac{n-5}{6},\\
4k>n-2k,{\hskip 0.6cm}&~\text{if~}\frac{n+1}{6}\leq k\leq \frac{n-1}{2},\\
4k<n+2-2k,&~\text{if~}1\leq k\leq \frac{n+1}{6},\\
4k>n+2-2k,&~\text{if~}\frac{n+7}{6}\leq k\leq \frac{n+1}{2}.
\end{aligned}
\right.
\end{align}
Note that
\begin{align*}
&\sum\limits_{1<2k<n}(\sum\limits_{l=1}^{\text{min}(4k,n-2k)}q^{-\frac{3}{2}-k+l})X_{n-1-4k} -\sum\limits_{1<2k<n-2}(\sum\limits_{l=1}^{\text{min}(4k,n-2-2k)}q^{-\frac{3}{2}-k+l})X_{n-1-4k}\\
=&\sum\limits_{k=\frac{n+1}{6}}^{\frac{n-3}{2}}(\sum\limits_{l=n-2k-1}^{n-2k}q^{-\frac{3}{2}-k+l})X_{n-1-4k} +q^{-\frac{n}{2}}X_{1-n}
\end{align*}
and
$$
\sum\limits_{l=1}^{\frac{2n+2}{3}}q^{-\frac{1}{2}-\frac{n+1}{6}+l}X_{\frac{n+7}{3}} -\sum\limits_{l=1}^{\frac{2n-1}{3}}q^{-\frac{1}{2}-\frac{n+1}{6}+l}X_{\frac{n+7}{3}}=q^{\frac{n}{2}}X_{\frac{n+7}{3}}.
$$

It follows that
\begin{align*}
&X_1X_{n+3}\\
=&\sum\limits_{k=\frac{n+7}{6}}^{\frac{n-1}{2}}(\sum\limits_{l=n+1-2k}^{n+2-2k}q^{-\frac{1}{2}-k+l})X_{n+3-4k} + q^{-\frac{n}{2}}X_{1-n}\\
&+\sum\limits_{1<2k<n}(\sum_{l=1}^{\text{min}(4k,n-2k)}q^{-\frac{1}{2}-k+l})X_{n+3-4k} +q^{\frac{n}{2}}X_{\frac{n+7}{3}} +q^{\frac{n+1}{2}}X_{\frac{n+4}{3}}X_{\frac{n+7}{3}}\\
=&\sum\limits_{1<2k<n+2}(\sum_{l=1}^{\text{min}(4k,n+2-2k)}q^{-\frac{1}{2}-k+l})X_{n+3-4k}+ q^{\frac{n+1}{2}}X_{\frac{n+4}{3}}X_{\frac{n+7}{3}}.
\end{align*}
This completes the proof of (\ref{equation4}). The proof of (\ref{equation5}) is similar to the proof of (\ref{equation4}).

(4) To prove (\ref{equation6}), it suffices to prove that
\begin{equation}\label{equ6}
X_1X_{1+2n}=q^{2n}X_{1+n}^{2\langle 1+n\rangle}+\sum\limits_{k=1}^{n-1}(\sum\limits_{l=1}^{4\text{min}(k,n-k)}q^{-\frac{1}{2}+l})X_{2n-2k+1}
+\sum\limits_{k\geq1}c_{n,k}F_{2n-2k}(X_{\delta}).
\end{equation}

When $n=1$, we have that $X_1X_3=q^2X^{4}_{2}+1$ which is the exchange relation. When $n=2$, note that
$
X_3F_{2}(X_\delta)=q^{-2}X_1+q^2X_5+(q^{-\frac{3}{2}}+ q^{-\frac{1}{2}}+q^{\frac{1}{2}}+q^{\frac{3}{2}})
$
by (\ref{equation2}), then we have that
\begin{align*}
&X_1X_5=q^{-2}X_1X_{3}F_{2}(X_\delta)-q^{-4}X_{1}^{2}- (q^{-\frac{7}{2}}+ q^{-\frac{5}{2}}+q^{-\frac{3}{2}}+q^{-\frac{1}{2}})X_1\\
=&q^{-1}X_{2}^{3}X_{-2}+qX_{2}^{3}X_{6}+q^{-2}F_{2}(X_\delta)-q^{-4}X_{1}^{2}-(q^{-\frac{7}{2}}+ q^{-\frac{5}{2}}+q^{-\frac{3}{2}}+q^{-\frac{1}{2}})X_1.
\end{align*}
Note that
\begin{enumerate}
\item[]$
q^{-1}X_{2}^{3}X_{-2}=q^{-2}X_{2}^{2}X_{0}^{2}+q^{-\frac{1}{2}}X_{2}^{2}X_{\delta},
$
\item[]$
qX_{2}^{3}X_{6}=q^{2}X_{2}^{2}X_{4}^{2}+q^{\frac{1}{2}}X_{2}^{2}X_{\delta},
$
\item[]$
X_{2}^{2}X_{0}^{2}=q^{-\frac{1}{2}}X_2X_1X_0+X_2X_0=q^{-2}X_{1}^{2}+(q^{-\frac{3}{2}}+q^{-\frac{1}{2}})X_1+1,
$
\item[]$
X_{2}^{2}X_{\delta}=X_2(q^{-\frac{1}{2}}X_0+q^{\frac{1}{2}}X_4)=q^{-1}X_1+qX_3+(q^{-\frac{1}{2}}+q^{\frac{1}{2}}),
$
\item[]$
X_{2}^{2}X_{4}^{2}=q^{\frac{1}{2}}X_2X_3X_4+X_2X_4=q^{2}X_{3}^{2}+(q^{\frac{1}{2}}+q^{\frac{3}{2}})X_3+1.
$
\end{enumerate}

It follows that
\begin{equation*}
X_1X_5=q^4X_{3}^{2}+(q^{\frac{1}{2}}+q^{\frac{3}{2}}+q^{\frac{5}{2}}+q^{\frac{7}{2}})X_3 +q^{-2}F_{2}(X_\delta)
+(q^{-2}+q^{-1}+2+q+q^{2}).
\end{equation*}

Assume that (\ref{equ6}) is true. We have that
\begin{align*}
X_1X_{3+2n}=q^{-2}X_1X_{1+2n}F_{2}(X_\delta)-q^{-4}X_1X_{2n-1}-(q^{-\frac{7}{2}}+ q^{-\frac{5}{2}}+q^{-\frac{3}{2}}+q^{-\frac{1}{2}})X_1.
\end{align*}

If $n$ is even, then
\begin{align*}
&q^{-2}X_1X_{1+2n}F_{2}(X_\delta)\\
=&q^{2n-2}X_{1+n}^{2}F_{2}(X_\delta) +\sum\limits_{k=1}^{n-1} (\sum\limits_{l=1}^{4\text{min}(k,n-k)}q^{-\frac{5}{2}+l})X_{2n+1-2k}F_{2}(X_\delta)\\
&+ \sum\limits_{k\geq1}q^{-2}c_{n,k}F_{2n-2k}(X_\delta)F_{2}(X_\delta).
\end{align*}

For convenience, we set
\begin{align*}
A=&q^{2n-2}X_{1+n}^{2}F_{2}(X_\delta),\\
B=& \sum\limits_{k=1}^{n-1} (\sum\limits_{l=1}^{4\text{min}(k,n-k)}q^{-\frac{5}{2}+l})X_{2n+1-2k}F_{2}(X_\delta),\\
C=&\sum\limits_{k\geq1}q^{-2}c_{n,k}F_{2n-2k}(X_\delta)F_{2}(X_\delta), \\
D=&q^{-4}X_1X_{2n-1},\\
E=& (q^{-\frac{7}{2}}+ q^{-\frac{5}{2}}+q^{-\frac{3}{2}}+q^{-\frac{1}{2}})X_1.
\end{align*}


Then we have
\begin{align*}
&A=q^{2n-2}X_{1+n}(q^{-2}X_{n-1}+q^2X_{n+3}+(q^{-\frac{3}{2}}+q^{-\frac{1}{2}}+q^{\frac{1}{2}}+q^{\frac{3}{2}}))\\
=&q^{2n-6}X_{n}^{4}+q^{2n+2}X_{n+2}^{4}+q^{2n-4}+q^{2n}
+ (q^{2n-\frac{7}{2}}+q^{2n-\frac{5}{2}}+q^{2n-\frac{3}{2}}+q^{2n-\frac{1}{2}})X_{n+1},
\end{align*}
\begin{align*}
&B=\sum\limits_{k=1}^{n-1} \sum\limits_{l=1}^{4\text{min}(k,n-k)}q^{-\frac{5}{2}+l} (q^{-2}X_{2n-1-2k}+q^2X_{2n+3-2k} +q^{-\frac{3}{2}}+q^{-\frac{1}{2}}+q^{\frac{1}{2}}+q^{\frac{3}{2}})\\
=&\sum\limits_{k=1}^{n-1} (\sum\limits_{l=1}^{4\text{min}(k,n-k)}q^{-\frac{9}{2}+l})X_{2n-1-2k} +\sum\limits_{k=1}^{n-1} (\sum\limits_{l=1}^{4\text{min}(k,n-k)}q^{-\frac{1}{2}+l})X_{2n+3-2k}\\
&+\sum\limits_{k=1}^{n-1} \sum\limits_{l=1}^{4\text{min}(k,n-k)}(q^{-4+l}+ q^{-3+l}+q^{-2+l}+q^{-1+l}),
\end{align*}
\begin{align*}
&C=\sum\limits_{k=1}^{n-2}c_{n+1,k}[F_{2n-2-2k}(X_\delta)+F_{2n+2-2k}(X_\delta)]+c_{n+1,n-1}(F_{4}(X_\delta)+2)+c_{n+1,n}F_2(X_\delta)
\end{align*}
and
$$
D=q^{2n-6}X_{n}^{4}+\sum\limits_{k=1}^{n-2}(\sum\limits_{l=1}^{\text{min}(k,n-1-k)}q^{-\frac{9}{2}+l})X_{2n-1-2k} +\sum\limits_{k\geq1}c_{n+1,k}F_{2n-2-2k}(X_\delta).
$$

It follows that
\begin{align*}
&X_1X_{3+2n}=A+B+C-D-E\\
=&q^{2n-4}+q^{2n}+q^{2n+2}X_{n+2}^{4}+(q^{2n-\frac{7}{2}}+q^{2n-\frac{5}{2}}+q^{2n-\frac{3}{2}} +q^{2n-\frac{1}{2}}) X_{n+1}\\
&+\sum\limits_{k=1}^{n-1}(\sum\limits_{l=1}^{4\text{min}(k,n-k)}q^{-\frac{9}{2}+l})X_{2n-1-2k} +\sum\limits_{k=1}^{n-1}(\sum\limits_{l=1}^{4\text{min}(k,n-k)}q^{-\frac{1}{2}+l})X_{2n+3-2k}\\
&+\sum\limits_{k=1}^{n-1}\sum\limits_{l=1}^{4\text{min}(k,n-k)}(q^{-4+l}+q^{-3+l}+q^{-2+l}+q^{-1+l}) +\sum\limits_{k=1}^{n-2}c_{n+1,k}F_{2n+2-2k}(X_\delta)\\
&+c_{n+1,n-1}(F_{4}(X_\delta)+1)+c_{n+1,n}F_2(X_\delta) -\sum\limits_{k=1}^{n-2}(\sum\limits_{l=1}^{4\text{min}(k,n-1-k)}q^{-\frac{9}{2}+l})X_{2n-1-2k}\\
&-(q^{-\frac{7}{2}}+q^{-\frac{5}{2}}+q^{-\frac{3}{2}}+q^{-\frac{1}{2}})X_1
\end{align*}
\begin{align*}
=&q^{2n-4}+q^{2n}+q^{2n+2}X_{n+2}^{4}+(q^{2n-\frac{7}{2}}+q^{2n-\frac{5}{2}}+q^{2n-\frac{3}{2}} +q^{2n-\frac{1}{2}}) X_{n+1}\\
&+\sum\limits_{k=1}^{n-2}(\sum\limits_{l=1}^{4\text{min}(k,n-k)}q^{-\frac{9}{2}+l})X_{2n-1-2k} +\sum\limits_{k=1}^{n-1}(\sum\limits_{l=1}^{4\text{min}(k,n-k)}q^{-\frac{1}{2}+l})X_{2n+3-2k}\\
&+\sum\limits_{k=1}^{n-1}\sum\limits_{l=1}^{4\text{min}(k,n-k)}(q^{-4+l}+q^{-3+l}+q^{-2+l}+q^{-1+l}) +\sum\limits_{k=1}^{n}c_{n+1,k}F_{2n+2-2k}(X_\delta)\\
&+c_{n+1,n-1} -\sum\limits_{k=1}^{n-2}(\sum\limits_{l=1}^{4\text{min}(k,n-1-k)}q^{-\frac{9}{2}+l})X_{2n-1-2k}.
\end{align*}

Since
\begin{align*}
\text{min}(k,n-1-k)=
\left\{
\begin{aligned}
k,{\hskip 1.3cm}&~\text{if~}k\leq \frac{n}{2}-1;\\
n-1-k,&~\text{if~} k\geq \frac{n}{2},
\end{aligned}
\right.
\end{align*}
\begin{align*}
\text{min}(k,n-k)=
\left\{
\begin{aligned}
k,{\hskip 0.6cm} &~\text{if~}k\leq \frac{n}{2}-1;\\
n-k,&~\text{if~} k\geq \frac{n}{2},
\end{aligned}
\right.
\end{align*}
and
\begin{align*}
\text{min}(k,n+1-k)=
\left\{
\begin{aligned}
k,{\hskip 1.3cm}&~\text{if~}k\leq \frac{n}{2};\\
n+1-k,&~\text{if~} k\geq \frac{n}{2}+1,
\end{aligned}
\right.
\end{align*}
we have that

\begin{align*}
&\sum\limits_{k=1}^{n-2}(\sum\limits_{l=1}^{4\text{min}(k,n-k)}q^{-\frac{9}{2}+l})X_{2n-1-2k} -\sum\limits_{k=1}^{n-2}(\sum\limits_{l=1}^{4\text{min}(k,n-1-k)}q^{-\frac{9}{2}+l})X_{2n-1-2k}\\
=&\sum\limits_{k=\frac{n}{2}}^{n-2}(\sum\limits_{l=1}^{4(n-k)}q^{-\frac{9}{2}+l})X_{2n-1-2k} -\sum\limits_{k=\frac{n}{2}}^{n-2}(\sum\limits_{l=1}^{4(n-1-k)}q^{-\frac{9}{2}+l})X_{2n-1-2k}\\
=&\sum\limits_{k=\frac{n}{2}}^{n-2}(\sum\limits_{l=4(n-k)-3}^{4(n-k)}q^{-\frac{9}{2}+l})X_{2n-1-2k}.
\end{align*}
Note that
\begin{align*}
&(q^{2n-\frac{7}{2}}+q^{2n-\frac{5}{2}}+q^{2n-\frac{3}{2}} +q^{2n-\frac{1}{2}}) X_{n+1} +\sum\limits_{k=\frac{n}{2}}^{n-2}(\sum\limits_{l=4(n-k)-3}^{4(n-k)}q^{-\frac{9}{2}+l})X_{2n-1-2k}\\
&+\sum\limits_{k=1}^{n-1}(\sum\limits_{l=1}^{4\text{min}(k,n-k)}q^{-\frac{1}{2}+l})X_{2n+3-2k}\\
=&\sum\limits_{k=\frac{n}{2}+1}^{n}(\sum\limits_{l=4(n-k)+1}^{4(n+1-k)}q^{-\frac{1}{2}+l})X_{2n+3-2k} +\sum\limits_{k=1}^{n-1}(\sum\limits_{l=1}^{4\text{min}(k,n-k)}q^{-\frac{1}{2}+l})X_{2n+3-2k}\\
=&\sum\limits_{k=1}^{n}(\sum\limits_{l=1}^{4\text{min}(k,n+1-k)}q^{-\frac{1}{2}+l})X_{2n+3-2k}.
\end{align*}
Therefore
\begin{align*}
&X_1X_{3+2n}\\
=&q^{2n+2}X_{n+2}^{4} +\sum\limits_{k=1}^{n}(\sum\limits_{l=1}^{4\text{min}(k,n+1-k)}q^{-\frac{1}{2}+l})X_{2n+3-2k} +\sum\limits_{k=1}^{n}c_{n+1,k}F_{2n+2-2k}(X_\delta) \\ &+q^{2n-4}+q^{2n}+\sum\limits_{k=1}^{n-1}\sum\limits_{l=1}^{4\text{min}(k,n-k)}(q^{-4+l}+q^{-3+l}+q^{-2+l}+q^{-1+l}) +c_{n+1,n-1}.
\end{align*}

We need to show that
\begin{align}\label{equation7}
q^{2n-4}+q^{2n}+\sum\limits_{k=1}^{n-1}\sum\limits_{l=1}^{4\text{min}(k,n-k)}(q^{-4+l}+q^{-3+l}+q^{-2+l}+q^{-1+l}) +c_{n+1,n-1}
=c_{n+1,n+1}.
\end{align}
Note that $c_{n+1,n-1}=\sum\limits_{i=1}^{n-1}a_i(q^{-2(n-i)-3}+q^{2(n-i)-5})+\sum\limits_{i=1}^{n-2}b_i(q^{-2(n-i)-2}+q^{2(n-i)-6}) +b_{n-1}q^{-4}$.

Let $P_n:=\sum\limits_{k=1}^{n-1}\sum\limits_{l=1}^{4\text{min}(k,n-k)}(q^{-4+l}+q^{-3+l}+q^{-2+l}+q^{-1+l})$.

We consider the coefficients of $q^{x}$ in the left hand side of (\ref{equation7}) for $0\leq x\leq 2n$. When $n=2$, we have that $c_{3,1}=q^{-4}$ and $P_2=q^{-3}+2q^{-2}+3q^{-1}+4+3q^{}+2q^{2}+q^{3}$. Thus $c_{3,3}=1+q^{4}+c_{3,1}+P_2$.

When $n=4$, $c_{5,3}=q^{-8}+q^{-7}+2q^{-6}+3q^{-5}+5q^{-4}+3q^{-3}+2q^{-2}+q^{-1}+1$ and $P_4= 3q^{-3}+6q^{-2}+9q^{-1}+12 +10q+8q^{2}+6q^{3}+4q^{4}+3q^{5}+2q^{6}+q^{7}$. Hence $c_{5,5}=q^{4}+q^8+c_{5,3}+P_4$.

Let $n\geq6$, we have that
\begin{enumerate}
\item[(i)]when $x\geq 2n-7$, the coefficients of $q^{x}$ in $c_{n+1,n-1}$ are $0$. The coefficients of $q^{x}$ in $P_n$ are $0$ for $x\geq2n$. For $0\leq x\leq 2n-4$, the coefficients of $q^{x}$ in the polynomial $P_n$ are $4n-4-2x$;
\item[(ii)]it is easy to see that the coefficients of $q^{2n-4}$, $q^{2n-3},q^{2n-2},q^{2n-1}$ and $q^{2n}$ in $P_n$ are $4$, $3$, $2$, $1$ and $0$, respectively. Thus the coefficients of $q^{2n-4},q^{2n-3},q^{2n-2},q^{2n-1}$ and $q^{2n}$ in the left hand side of (\ref{equation7}) are $5,3,2,1$ and $1$, respectively;
\item[(iii)] when $0\leq x\leq 2n-8$ and $x\equiv0~(\text{mod}~4)$, the coefficients of $q^{x}$ in $c_{n+1,n-1}$ are $b_{n-3-\frac{x}{2}}$, then the coefficients of $q^{x}$ in the left hand side of (\ref{equation7}) are $b_{n-3-\frac{x}{2}}+(4n-4-2x)=b_{n+1-\frac{x}{2}}$ since $b_{n-3-\frac{x}{2}}=\frac{1}{2}(n^2-(6+x)n+(\frac{1}{4}x^{2}+3x+10))$ and $b_{n+1-\frac{x}{2}} =\frac{1}{2}(n^2+(2-x)n+(\frac{1}{4}x^{2}-x+2))$;
\item[(iv)]when $0\leq x\leq 2n-8$ and $x\equiv2~(\text{mod}~4)$, the coefficients of $q^{x}$ in $c_{n+1,n-1}$ are $b_{n-3-\frac{x}{2}}$, then the coefficients of $q^{x}$ in the left hand side of (\ref{equation7}) are $b_{n-3-\frac{x}{2}}+(4n-4-2x)=b_{n+1-\frac{x}{2}}$ since $b_{n-3-\frac{x}{2}}=\frac{1}{2}(n^2-(6+x)n)+(\frac{1}{4}x^2 +3x+9)$ and $b_{n+1-\frac{x}{2}} =\frac{1}{2}(n^2+(2-x)n+(\frac{1}{4}x^{2}-x+1))$;
\item[(v)]when $1\leq x\leq 2n-7$ and $x$ are odd, the coefficients of $q^{x}$ in $c_{n+1,n-1}$ are $a_{n-\frac{x+5}{2}}$, so the coefficients of $q^{x}$ in the left hand side of (\ref{equation7}) are $a_{n-\frac{x+5}{2}}+(4n-4-2x)=\frac{1}{2}(n^{2}-(x+6)n+(\frac{1}{4}x^2+3x+\frac{35}{4})) +(4n-4-2x) =\frac{1}{2}(n^{2}+(2-x)n+(\frac{1}{4}x^2-x+\frac{3}{4})) =a_{n+\frac{3-x}{2}}$;
\item[(vi)]the coefficient of $q^{2n-6}$ in $P_n$ and $c_{n+1,n-1}$ are $8$ and $0$ respectively. It follows that the coefficient of $q^{2n-6}$ in the left hand side of (\ref{equation7}) is $b_4=8$. The coefficient of $q^{2n-5}$ in $P_n$ and $c_{n+1,n-1}$ are $6$ and $0$ respectively. Thus the coefficient of $q^{2n-5}$ in the left hand side of (\ref{equation7}) is $a_4=6$.

\end{enumerate}

It follows that the coefficients of $q^{x}$ in the left hand side of (\ref{equation7}) are
\begin{align*}
\left\{
\begin{aligned}
b_{n+1-\frac{x}{2}},&~\text{if}~ x\in\{0,2,\ldots,2n\},\\
a_{n-\frac{x-3}{2}},&~\text{if}~ x\in\{1,3,\ldots,2n-1\}.
\end{aligned}
\right.
\end{align*}

We consider the coefficients of $q^{-x}$ for $1\leq x\leq 2n$.
We have that
\begin{enumerate}
\item[(i)]when $4\leq x\leq 2n$, the coefficients of $q^{-x}$ in the polynomial $P_n$ are $0$;
\item[(ii)]when $x=1,2,3$, the coefficients of $q^{-1},q^{-2},q^{-3}$ in the polynomial $P_n$ are $3(n-1)$, $2(n-1)$ and $n-1$ respectively. Note that the coefficients of $q^{-1},q^{-2},q^{-3}$ in $c_{n+1,n-1}$ are $a_{n-2},b_{n-2}$ and $a_{n-1}$ respectively. Hence the coefficients of $q^{-1},q^{-2},q^{-3}$ in the left hand side of (\ref{equation7}) are $a_{n+1},b_n$ and $a_n$ respectively;
\item[(iii)]when $x=4,6,8,\ldots,2n$, the coefficients of $q^{-x}$ in $c_{n+1,n-1}$ are $b_{n+1-\frac{x}{2}}$. Thus the coefficients of $q^{-x}$ in the left hand side of (\ref{equation7}) are $b_{n+1-\frac{x}{2}}$ for $x=4,6,8,\ldots,2n$;
\item[(iv)]when $x=5,7,\ldots,2n-1$, the coefficients of $q^{-x}$ in $c_{n+1,n-1}$ are $a_{n-\frac{x-3}{2}}$. We obtain that the coefficients of $q^{-x}$ in the left hand side of (\ref{equation7}) are $a_{n-\frac{x-3}{2}}$ for $x=5,7,\ldots,2n-1$.

\end{enumerate}

Hence the coefficients of $q^{-x}$ in the left hand side of (\ref{equation7}) are
\begin{align*}
\left\{
\begin{aligned}
b_{n+1-\frac{x}{2}},&~\text{if}~ x\in\{2,\ldots,2n\},\\
a_{n-\frac{x-3}{2}}, &~\text{if}~ x\in\{1,3,\ldots,2n-1\}.
\end{aligned}
\right.
\end{align*}
The identity (\ref{equation7}) is true. This proceeds the induction and finishes the proof.

If $n$ is odd, by a similar detailed calculation, we can proceed the proof. Thus we can draw the conclusion that (\ref{equation6}) is true.

The entire proof of the theorem is now completed.
\end{proof}
\begin{remark}
According to the definition of bar-invariant and Theorem ~\ref{theorem2}, we can easily obtain
the similar cluster multiplication formulas for $F_n(X_{\delta})X_{m}$ and $X_{n}X_{m}$ in other cases.
\end{remark}

As an immediate consequence of the above theorem, we can naturally obtain  the following result proved in ~\cite{cds} by a simple way.
\begin{corollary}
The sets $\mathcal{B}$, $\mathcal{S}$ and $\mathcal{D}$ are bar-invariant $\ZZ[q^{\pm\frac{1}{2}}]$-bases of $\A_{q}(1,4)$.
\end{corollary}
\begin{proof}
Since there exist unipotent transformations between $F_{n}(X_\delta)$, $S_{n}(X_\delta)$ and $X^{n}_\delta$ for $n\geq 1$, we only need to prove that the set $\mathcal{B}$ is a
bar-invariant $\ZZ[q^{\pm\frac{1}{2}}]$-basis of $\A_{q}(1,4)$. It suffices to show that $\mathcal{B}$ are linearly independent.  Note that the denominator vectors of $F_n(X_\delta)$ in $\A_q(1,4)$ are $(n,2n)$ for $n\in\mathbb{N}$ which are bijective to  the set of positive imaginary roots  of the corresponding Lie algebra denoted by $\Phi_{+}^{\rm{im}}$. By \cite[Proposition 3.1]{SZ}, there exists a bijection between the set of all denominator vectors of cluster monomials and $\mathcal{Q}-\Phi_{+}^{\rm{im}}$  with $\mathcal{Q}:=\mathbb{Z}^2$ being a lattice of rank
2 with a fixed basis of two simple roots $\{\alpha_1,\alpha_2\}$. Thus the denominator vectors of cluster monomials and $F_n(X_\delta)$ are different from each other. Assume that $S$ is a finite set and
$$\sum\limits_{\alpha\in S} a_\alpha X_\alpha=0$$
for $X_\alpha\in\mathcal{B}$ and $a_\alpha\in\ZZ[q^{\pm\frac{1}{2}}]\setminus\{0\}$. Let $V(S)$ denote the set of denominator vectors of $X_\alpha$ for $\alpha\in S$ and let $\leq$ be the partial order on $V(S)$ inherited from $\mathbb{Z}^2$. There exists $\beta\in S$ such that the denominator vector of $X_\beta$ is a maximal denominator vector in $(V(S),\leq)$. We can then deduce that $a_\beta=0$ which is a contradiction. Therefore $\mathcal{B}$ is linearly independent.
\end{proof}
Recall that an element  $Y\in\A_q(1,4)$ is \emph{positive} if the coefficients in its Laurent expansion in the cluster variables from $\{X_m,X_{m+1}\}$ belong to $\NN[q^{\pm\frac{1}{2}}]$.
The following result can be deduced from Theorem \ref{theorem2}.
\begin{corollary}\label{corbasepos}
The elements in the $\ZZ[q^{\pm\frac{1}{2}}]$-bases $\mathcal{B}$, $\mathcal{S}$ and $\mathcal{D}$ of $\A_{q}(1,4)$ are positive.
\end{corollary}
\begin{proof}
By the definitions of Chebyshev polynomials $F_{n}(x)$ and $S_{n}(x)$, we only need to prove the positivity of the elements in $\mathcal{B}$. Using the fact that $\sigma_2$ is an automorphism, it suffices to prove the positivity in the clusters $\{X_1,X_2\}$ and $\{X_2,X_3\}$. Note that
\begin{align*}
X_0=&X^{(1,-1)}+X^{(0,-1)},\\
X_{-1}=&X^{(3,-4)}+(q^{-\frac{3}{2}}+q^{-\frac{1}{2}}+q^{\frac{1}{2}}+q^{\frac{3}{2}})X^{(2,-4)} +(q^{-2}+q^{-1}+1+q^{}+q^{2})X^{(1,-4)} \\
&+(q^{-\frac{3}{2}}+q^{-\frac{1}{2}}+q^{\frac{1}{2}}+q^{\frac{3}{2}})X^{(0,-4)} +X^{(-1,0)}+X^{(-1-4)},\\
X_{-2}=&X^{(-1,1)}+X^{(2,-3)}+(q^{-1}+1+q)X^{(1,-3)}+(q^{-1}+1+q)X^{(0,-3)}+X^{(-1,-3)},\\
X_\delta=&X^{(-1,-2)}+X^{(-1,2)}+X^{(1,-2)}+(q^{-\frac{1}{2}}+q^{\frac{1}{2}})X^{(0,-2)},\\
F_2(X_\delta)=&X^{(-2,-4)}+(q^{-2}+q^{2})X^{(-2,0)}+(q^{-2}+q^{-1}+2+q+q^{2})X^{(0,-4)}\\
&+(q^{-\frac{3}{2}}+q^{-\frac{1}{2}}+q^{\frac{1}{2}}+q^{\frac{3}{2}})(X^{(-1,4)}+X^{(1,-4)}+X^{(-1,0)})+X^{(-2,4)} +X^{(2,-4)}.
\end{align*}

Hence $X_{-2},X_{-1},X_0,X_1,X_2,X_\delta$ and $F_2(X_\delta)$ are positive elements in $\{X_1,X_2\}$. Suppose that $X_{-n},X_{-n+1},\ldots,X_{-2},X_{-1},X_0,X_1,X_2,\ldots,X_{n-1},X_{n},X_\delta,F_2(X_\delta),\ldots,$ and $F_{n}(X_\delta)$ are positive. By Theorem \ref{theorem2}, when $n$ is odd, we have
\begin{align*}
X_1X_{1+n}=&\sum\limits_{1<2k<n}(\sum\limits_{l=1}^{\text{min}(4k,n-2k)}q^{-\frac{1}{2}-k+l})X_{n+1-4k}
+\left\{
\begin{aligned}
q^{\frac{n}{2}}X^{3}_{1+\frac{n}{3}},{\hskip 1.8cm}&~n\equiv 0~(\rm{mod}~3);~~~~~~~ \\
q^{\frac{n-1}{2}}X_{\lfloor 1+\frac{n}{3}\rfloor}X_{\lceil 1+\frac{n}{3}\rceil},&~\rm{otherwise,}
\end{aligned}
\right.
\end{align*}

\begin{align*}
&X_1X_{-n-1}\\=&\sum\limits_{1<2k<n+2}(\sum\limits_{l=1}^{\text{min}(4k,n+2-2k)}q^{\frac{1}{2}+k-l})X_{-n-1+4k}+\left\{
\begin{aligned}
q^{-\frac{n+2}{2}}X^{3}_{\frac{1-n}{3}},{\hskip 1.8cm}&~n\equiv 1~(\rm{mod}~3);~~~~~~~ \\
q^{-\frac{n+3}{2}}X_{\lfloor \frac{1-n}{3}\rfloor}X_{\lceil \frac{1-n}{3}\rceil},&~\rm{otherwise,}
\end{aligned}
\right.
\end{align*}
$$
X_1X_{2+n}=q^{n+1}X^{2\langle\frac{n+3}{2}\rangle}_{\frac{n+3}{2}}+\sum\limits_{k=1}^{\frac{n-1}{2}} (\sum\limits_{l=1}^{4\text{min}(k,\frac{n+1}{2}-k)}q^{-\frac{1}{2}+l})X_{n+2-2k}+\sum\limits_{k\geq1}c_{\frac{n+1}{2},k} F_{n+1-2k}(X_\delta)
$$
and
$$
X_{-n-2}X_{1}=q^{n+3}X^{2\langle-\frac{n+1}{2}\rangle}_{-\frac{n+1}{2}}+\sum\limits_{k=1}^{\frac{n+1}{2}} (\sum\limits_{l=1}^{4\text{min}(k,\frac{n+3}{2}-k)}q^{-\frac{1}{2}+l})X_{1-2k}+ \sum\limits_{k\geq1}c_{\frac{n+3}{2},k} F_{n+3-2k}(X_\delta);
$$
when $n$ is even, we have

\begin{align*}
X_1X_{1+n}=&q^{n}X^{2\langle\frac{n}{2}+1\rangle}_{\frac{n}{2}+1}+\sum\limits_{k=1}^{\frac{n-2}{2}} (\sum\limits_{l=1}^{4\text{min}(k,\frac{n}{2}-k)}q^{-\frac{1}{2}+l})X_{n+1-2k}+\sum\limits_{k\geq1}c_{\frac{n}{2},k} F_{n-2k}(X_\delta),\\
X_{-n-1}X_{1}=&q^{n+2}X^{2\langle-\frac{n}{2}\rangle}_{-\frac{n}{2}}+\sum\limits_{k=1}^{\frac{n}{2}} (\sum\limits_{l=1}^{4\text{min}(k,\frac{n}{2}+1-k)}q^{-\frac{1}{2}+l})X_{1-2k}+ \sum\limits_{k\geq1}c_{\frac{n}{2}+1,k} F_{n+2-2k}(X_\delta),\\
X_2X_{2+n}=&q^{\frac{n}{4}}X_{2+\frac{n}{2}}^{\langle 2+\frac{n}{2}\rangle}+\sum\limits_{k\geq1}(\sum\limits_{l=1}^{2k-1} q^{-\frac{n+2}{4}+l})F_{\frac{n}{2}+1-2k}(X_\delta),\\
X_{-n-2}X_{2}=&q^{\frac{n+4}{4}}X_{-\frac{n}{2}}^{\langle-\frac{n}{2}\rangle}+\sum\limits_{k\geq1}(\sum\limits_{l=1}^{2k-1} q^{-\frac{n+6}{4}+l})F_{\frac{n}{2}+3-2k}(X_\delta).
\end{align*}

We deduce that $X_{-n-2},X_{-n-1},X_{n+1}$ and $X_{n+2}$ are positive in $\{X_1,X_2\}$. According to Theorem \ref{theorem2}, we have that
\begin{align*}
& X_1F_{n+1}(X_\delta)\\=&q^{-(n+1)}X_{-n}^{\langle -n\rangle}+q^{n+1}X_{n+2}^{\langle n+2\rangle}\nonumber+\sum\limits_{k\geq1} (\sum\limits_{l=1}^{k} (q^{-\frac{4l-1}{2}}+q^{-\frac{4l-3}{2}}+q^{\frac{4l-3}{2}}+q^{\frac{4l-1}{2}}))F_{n+1-2k}(X_{\de}).
\end{align*}
It follows that $F_{n+1}(X_\delta)$ is positive in $\{X_1,X_2\}$. By induction, each element in $\mathcal{B}$ is positive in $\{X_1,X_2\}$. Similarly, each element in $\mathcal{B}$ is positive in $\{X_2,X_3\}$. The proof is completed.
\end{proof}

A $\ZZ[q^{\pm\frac{1}{2}}]$-basis $\mathcal{C}$ of $\A_q(1,4)$ is called the \emph{canonical} basis of $\A_q(1,4)$ if every positive element of $\A_q(1,4)$ is the $\NN[q^{\pm\frac{1}{2}}]$-linear combination of elements of $\mathcal{C}$.
\begin{theorem}
The basis $\mathcal{B}$ is the canonical basis of $\A_q(1,4)$.
\end{theorem}
\begin{proof}
Let $Y$ be a positive element of $\A_q(1,4)$. Since $\mathcal{B}$ is a $\ZZ[q^{\pm\frac{1}{2}}]$-basis of $\A_q(1,4)$, we have $Y=\sum\limits_{Z\in\mathcal{B}}a_ZZ$ for $a_Z\in\ZZ[q^{\pm\frac{1}{2}}]\setminus\{0\}$. It suffices to prove that each $a_Z$ equals some coefficient in the Laurent expansion of $Y$ with respect to $X_m$ and $X_{m+1}$ for $m\in\ZZ$.

First of all, we consider the coefficient of the cluster monomial $q^{-\frac{1}{2}ab}X_{m}^{a}X_{m+1}^{b}$, where $a,b\in\NN$. Since $\sigma_{2}$ is an automorphism, it suffices to show that the coefficient $a_{X^{(c,d)}}$ of $X^{(c,d)}$ in $Y=\sum\limits_{Z\in\mathcal{B}}a_ZZ$ equals the coefficient of $X^{(c,d)}$ in the Laurent expansion of $Y$ with respect to $\{X_1,X_{2}\}$ (showing that the coefficient of the cluster monomial $q^{-\frac{1}{2}cd}X_{2}^{c}X_{3}^{d}$ in $Y=\sum\limits_{Z\in\mathcal{B}}a_ZZ$ equals the coefficient of $q^{-\frac{1}{2}cd}X_{2}^{c}X_{3}^{d}$ in the Laurent expansion of $Y$ with respect to $\{X_2,X_3\}$ uses the same method). For $m\neq1$, let
$$
q^{-\frac{1}{2}ab}X_{m}^{a}X_{m+1}^{b}=\sum\limits_{(c,d)\in\ZZ^{2}}M_{cd}(q)X^{(c,d)}
$$
be the Laurent expansion of the cluster monomial $q^{-\frac{1}{2}ab}X_{m}^{a}X_{m+1}^{b}$ with respect to $X_1$ and $X_{2}$, where $M_{cd}(q)\in\ZZ[q^{\pm\frac{1}{2}}]$. For $n\geq1$, let
$$
F_n(X_\delta)=\sum\limits_{(e,f)\in\ZZ^{2}}N_{ef}(q)X^{(e,f)}
$$
be the Laurent expansion of $F_n(X_\delta)$ with respect to $X_1$ and $X_2$, where $N_{ef}(q)\in\ZZ[q^{\pm\frac{1}{2}}]$. By Corollary \ref{corbasepos}, $M_{cd}(q)$, $N_{ef}(q)\in\NN[q^{\pm\frac{1}{2}}]$. We only need to show that if $c,d,e,f\in\NN$ then $M_{cd}(q)=N_{ef}(q)=0$. If there exists some $M_{cd}(q)\neq0$ for some $c,d\in\NN$, then $M_{cd}(1)>0$ which contradicts the statement in \cite[Proposition 3.6]{SZ}. Thus $M_{cd}(q)=0$ for $c,d\geq0$. If there exists some $N_{ef}(q)\neq0$ for $e,f\in\NN$, then $N_{ef}(1)>0$ which contradicts the statement in \cite[Proposition 5.2(2)]{SZ}. Hence $N_{ef}(q)=0$ for $e,f\geq0$.

It remains to consider the coefficient of $F_n(X_\delta)$ in $Y=\sum\limits_{Z\in\mathcal{B}}a_ZZ$ for $n\geq1$. Without loss of generality, we can assume that the cluster variables $X_m$ occurring in $Y=\sum\limits_{Z\in\mathcal{B}}a_ZZ$ satisfy that $m\geq3$ since $\sigma_{2}$ is an automorphism. It suffices to show that, the coefficient of the cluster monomial $X^{(n,-2n)}$ ($n\geq1$) in the Laurent expansion of $F_n(X_\delta)$ with respect to the cluster $\{X_1,X_2\}$ is $1$, but the coefficients of $X^{(n,-2n)}$ in the Laurent expansions of $F_k(X_\delta)$ or $q^{-\frac{1}{2}n_1n_2}X_{m}^{n_1}X_{m+1}^{n_2}$ for $k\neq n$ and $m\geq3$ are $0$. Let $R_1(q)$ denote the coefficient of $X^{(n,-2n)}$ in the Laurent expansion of $F_k(X_\delta)$ with respect to $\{X_1,X_2\}$. Note that $R_1(q)\in\NN[q^{\pm\frac{1}{2}}]$ since $F_n(X_\delta)$ is positive. Using the fact that $F_n(X_\delta)$ is bar-invariant and $R_1(1)=1$ by \cite[Proposition 5.1]{SZ}, it follows that $R_1(q)=1$. Let $R_2(q)\in\NN[q^{\pm\frac{1}{2}}]$ denote the coefficient of $X^{(n,-2n)}$ in the Laurent expansion of $F_k(X_\delta)$ with respect to $\{X_1,X_2\}$. If $R_2(q)\neq0$, then $R_2(1)\neq0$, which contradicts \cite[Proposition 5.2(1)]{SZ}. Thus $R_2(q)=0$. Let $R_3(q)\in\NN[q^{\pm\frac{1}{2}}]$ denote the coefficient of $X^{(n,-2n)}$ in the Laurent expansion of $q^{-\frac{1}{2}}X_{m}^{n_1}X_{m+1}^{n_2}$ with respect to $\{X_1,X_2\}$. If $R_3(q)\neq0$, then $R_3(1)\neq0$, which contracts the statement in \cite[Corollary 3.7]{SZ}. Hence $R_3(q)=0$. This finishes the proof.
\end{proof}

\section{triangular bases}
In \cite{BZ2014}, the choice of the exchange matrix and the skew-symmetric matrix are $B=\left(
  \begin{array}{cc}
    0 & -b \\
    c & 0 \\
  \end{array}
\right)$ for $b,c>0$ and $\Lambda=\left(
  \begin{array}{cc}
    0 & -1 \\
    1 & 0 \\
  \end{array}
\right)$, respectively. Though our exchange matrix $B$ and the skew-symmetric matrix $\Lambda$ differ from those of \cite{BZ2014} by a sign, let $v=q^{-\frac{1}{2}}$ to reconcile the formal variables $q$ and $v$, then Theorem \ref{theorem2} is still true.

The following theorem is a version of Lusztig's Lemma.
\begin{theorem}\cite[Theorem 1.1]{BZ2014}.\label{Lusztiglem}
Let $(L,\prec)$ be a partially order set and satisfy that the lengths of chains with the top element $u\in L$ in $L$ are bounded from above. Let $\A$ be a free $\ZZ[q^{\pm\frac{1}{2}}]$-module spanned by the basis $\{E_u~|~u\in L\}$. The bar-involution on $\A$ is the $\ZZ$-linear map $x\mapsto \overline{x}$ such that $\overline{fx}=\overline{f}\overline{x}$ for $f\in\ZZ[q^{\pm\frac{1}{2}}]$ and $x\in\A$, where $\overline{f}(q^{\frac{1}{2}})=f(q^{-\frac{1}{2}})$. If
\begin{equation}\label{conditionlusztig}
\bar{E}_{u}-E_u\in \bigoplus\limits_{u^\prime\prec u}\ZZ[q^{\pm\frac{1}{2}}]E_{u^\prime}
\end{equation}
for $u\in L$, then there exists a unique element $C_u\in\A$ satisfies that
\begin{equation}\label{equlusztiglem1}
\bar{C}_u=C_u,
\end{equation}
\begin{equation}\label{equlusztiglem2}
C_u-E_u\in\bigoplus\limits_{u^\prime\in L}q^{-\frac{1}{2}}\ZZ[q^{-\frac{1}{2}}]E_{u^\prime}.
\end{equation}
More precisely, (\ref{equlusztiglem2}) can be replaced by
\begin{equation}\label{equlusztiglem3}
C_u-E_u\in\bigoplus\limits_{u^\prime\prec u}q^{-\frac{1}{2}}\ZZ[q^{-\frac{1}{2}}]E_{u^\prime}.
\end{equation}
Then $\{C_u~|~u\in L\}$ is a $\ZZ[q^{\pm\frac{1}{2}}]$-basis in $\A$.
\end{theorem}

Let $\A$ be the quantum cluster algebra associated with an acyclic quantum seed. By \cite[Theorem 1.4, Theorem 1.6]{BZ2014}, the basis $\{C_u~|~u\in L\}$ of $\A$ is uniquely determined by (\ref{equlusztiglem1}) and (\ref{equlusztiglem2}). The basis $\{C_u~|~u\in L\}$ is called the \emph{canonical triangular} basis in $\A$ (see \cite[Section 1]{BZ2014} for more detail).

For $x\in\ZZ$, the function $[x]_+$ is defined by $[x]_+=x$ if $x>0$ and $[x]_+=0$ otherwise.

In the rest of this section, we set $L=\ZZ^2$ and $\A=\A_q(1,4)$. For $(a,b),(a^\prime,b^\prime)\in\ZZ^{2}$, recall that the partial order $\prec$ on $\ZZ^{2}$ used in \cite[(6.16)]{BZ2014} is defined by:
\begin{equation}\label{partialorder}
(a^\prime,b^\prime)\prec(a,b)\Longleftrightarrow[-a^\prime]_+<[-a]_+,~[-b^{\prime}]_+<[-b]_+.
\end{equation}
Using the same argument as in \cite[Section 2]{BZ2014}, it follows that $\A_q(1,4)$ satisfies the condition (\ref{conditionlusztig}) for the partial order $\prec$.

The standard monomials in $\A_q(1,4)$ are defined by
\begin{equation}\label{standardmonomial}
E_{(a,b)}=q^{-\frac{1}{2}ab}X_{3}^{[-a]_+}X_{1}^{[a]_+}X_{2}^{[b]_+}X_{0}^{[-b]_+}
\end{equation}
for $a,b\in\ZZ$ (\cite[Section 1]{BZ2014}). As in \cite{BZ2014}, the \emph{crystal lattice} $\A_+\subset\A_q(1,4)$ is defined by
$$
\A_+=\bigoplus\limits_{a,b\in\ZZ}\ZZ[q^{-\frac{1}{2}}]E_{(a,b)}.
$$

Recall that in \cite{BZ2014} for the the quantum cluster algebra of Kronecker type $\A_q(2,2)$, for $a,b\in\ZZ$, Berenstein and Zelevinsky defined the elements $$
E^{\prime}_{(a,b)}: =q^{-\frac{1}{2}ab}X_{2}^{[-b]_+}X_{0}^{[b]_+}X_{1}^{[a]_+}X_{-1}^{[-a]_+} \in\A_q(2,2).$$
 Similarly, we define the elements $E^{\prime}_{(a,b)}$ and $\mu_1E_{(a,b)}$ in $\A_q(1,4)$ as follows
$$
E^{\prime}_{(a,b)}=q^{-\frac{1}{2}ab}X_{2}^{[-b]_+}X_{0}^{[b]_+}X_{1}^{[a]_+}X_{-1}^{[-a]_+}~\text{and}~
\mu_1E_{(a,b)}=q^{-\frac{1}{2}ab}X_{4}^{[-b]_+}X_{2}^{[b]_+}X_{3}^{[a]_+}X_{1}^{[-a]_+}.
$$

A new partial order $\preceq$ on $\ZZ^{2}$ is defined by:
\begin{equation*}\label{partialorder}
(a^\prime,b^\prime)\preceq (a,b)\Longleftrightarrow[-a^\prime]_+\leq[-a]_+,~[-b^{\prime}]_+\leq[-b]_+.
\end{equation*}

\begin{remark}
The partial order $\preceq$ is used to prove Lemma \ref{lem4.1} and Lemma \ref{lempsi} and the partial order $\prec$ is used in Lusztig's Lemma.
\end{remark}

\begin{lemma}\label{lem4.1}
For $a,b\in\ZZ$, we have that
\begin{enumerate}
  \item $q^{\frac{1}{2}a}E_{(a,b)}X_0-E_{(a,b-1)}\in \bigoplus\limits_{(a^{\prime},b^{\prime})\preceq(a+1,b-1)}q^{-\frac{1}{2}}\ZZ[q^{-\frac{1}{2}}] E_{(a^{\prime},b^{\prime})}$;
  \item $q^{\frac{1}{2}b}X_3E_{(a,b)}-E_{(a-1,b)}\in \bigoplus\limits_{(a^{\prime},b^{\prime})\preceq(a-1,b+4)}q^{-\frac{1}{2}}\ZZ[q^{-\frac{1}{2}}] E_{(a^{\prime},b^{\prime})}$;
  \item $q^{\frac{1}{2}b}E_{(a,b)}X_1-E_{(a+1,b)}\in \bigoplus\limits_{(a^{\prime},b^{\prime})\preceq(a+1,b+4)}q^{-\frac{1}{2}}\ZZ[q^{-\frac{1}{2}}] E_{(a^{\prime},b^{\prime})}$;
  \item $q^{-\frac{1}{2}a}X_4E_{(a,b)}-E_{(a,b-1)}\in \bigoplus\limits_{(a^{\prime},b^{\prime})\preceq(a-1,b-1)}q^{-\frac{1}{2}}\ZZ[q^{-\frac{1}{2}}] E_{(a^{\prime},b^{\prime})}$ if $b>0$;
  \item $q^{-\frac{1}{2}(a-b)}X_4E_{(a,b)}-E_{(a-1,b-1)}\in \bigoplus\limits_{(a^{\prime},b^{\prime})\preceq(a-1,b+3)}q^{-\frac{1}{2}}\ZZ[q^{-\frac{1}{2}}] E_{(a^{\prime},b^{\prime})}$ if $a>0$ and $b\leq0$;
  \item $q^{-\frac{1}{2}(a-b)}X_4E_{(a,b)}-E_{(a-1,b-1)}\in \bigoplus\limits_{(a^{\prime},b^{\prime})\preceq(a,b)}q^{-\frac{1}{2}}\ZZ[q^{-\frac{1}{2}}] E_{(a^{\prime},b^{\prime})}$ if $a\leq0$ and $b\leq0$.
\end{enumerate}
\end{lemma}
\begin{proof}
(1) (i) If $a\geq0,b>0$, then
\begin{equation}\label{equ4.1.1i}
q^{\frac{1}{2}a}E_{(a,b)}X_0-E_{(a,b-1)}=q^{-\frac{1}{2}b}E_{(a+1,b-1)}.
\end{equation}

(ii) If $a\geq0,b\leq0$, then
\begin{equation}\label{equ4.1.1ii}
q^{\frac{1}{2}a}E_{(a,b)}X_0=q^{-\frac{1}{2}(ab-b)}X_{1}^{a}X_{0}^{-b+1}=E_{(a,b-1)}.
\end{equation}

(iii) If $a<0,b\leq0$, then
\begin{equation}\label{equ4.1.1iii}
q^{\frac{1}{2}a}E_{(a,b)}X_0=q^{-\frac{1}{2}(ab-b)}X_{3}^{-a}X_{0}^{-b+1}=E_{(a,b-1)}.
\end{equation}

(iv) If $a<0,b>0$, then
\begin{equation}\label{equ4.1.1iv}
q^{\frac{1}{2}a}E_{(a,b)}X_0-E_{(a,b-1)}=q^{-\frac{1}{2}b}E_{(a+1,b-1)}+q^{-\frac{1}{2}(b-4a)}E_{(a+1,b+3)}.
\end{equation}

(2) (i) If $a\leq0,b\geq0$, then
\begin{equation}\label{equ4.1.2i}
q^{\frac{1}{2}b}X_3E_{(a,b)}=q^{-\frac{1}{2}(ab-b)}X_{3}^{1-a}X_{2}^{b}=E_{(a-1,b)}.
\end{equation}

(ii) If $a\leq0,b<0$, then
\begin{equation}\label{equ4.1.2ii}
q^{\frac{1}{2}b}X_3E_{(a,b)}=q^{-\frac{1}{2}(a-1)b}X_{3}^{1-a}X_{0}^{-b}=E_{(a-1,b)}.
\end{equation}

(iii) If $a>0,b\geq0$, then
\begin{equation}\label{equ4.1.2iii}
q^{\frac{1}{2}b}X_3E_{(a,b)}-E_{(a-1,b)}=q^{-\frac{1}{2}(ab+8a-b-4)}X_{1}^{a-1}X_{2}^{b+4}=q^{-2a}E_{(a-1,b+4)}.
\end{equation}

(iv) If $a>0,b=-1$, then
\begin{equation}\label{equ4.1.2iv}
q^{-\frac{1}{2}}X_3E_{(a,-1)}-E_{(a-1,-1)}=q^{-2a}E_{(a-1,3)}+q^{-2a-2}E_{(a,3)}.
\end{equation}

(v) If $a>0,b=-2$, then
\begin{align}\label{equ4.1.2v}
q^{-1}X_3E_{(a,-2)}-E_{(a-1,-2)}=q^{-2a}E_{(a-1,2)}+q^{-2a}(q^{-\frac{5}{2}}+q^{-\frac{3}{2}})E_{(a,2)} +q^{-2a-4}E_{(a+1,2)}.
\end{align}

(vi) If $a>0,b=-3$, then
\begin{align}\label{equ4.1.2vi}
&q^{-\frac{3}{2}}X_3E_{(a,-3)}-E_{(a-1,-3)}\nonumber \\
=&q^{-2a}E_{(a-1,1)}+q^{-2a}(q^{-3}+q^{-2}+q^{-1})E_{(a,1)}+q^{-2a}(q^{-5}+q^{-4}+q^{-3})E_{(a+1,1)} \nonumber\\
& +q^{-2a-6}E_{(a+2,1)}.
\end{align}

(vii) If $a>0,b\leq-4$, then
\begin{align}\label{equ4.1.2vii}
&q^{\frac{1}{2}b}X_3E_{(a,b)}-E_{(a-1,b)} \nonumber\\
=&q^{-2a}E_{(a-1,b+4)} +q^{-\frac{1}{2}(4a-b-3)}(q^{-3}+q^{-2}+q^{-1}+1)E_{(a,b+4)} \nonumber\\
&+q^{-2a+b+2}(q^{-4}+q^{-3}+2q^{-2}+q^{-1}+1)E_{(a+1,b+4)} \nonumber\\
&+q^{-\frac{1}{2}(4a-3b-3)}(q^{-3}+q^{-2}+q^{-1}+1)E_{(a+2,b+4)}+ q^{-2a+2b}E_{(a+3,b+4)}.
\end{align}

(3) (i) If $a\geq0,b\geq0$, then
\begin{equation}\label{equ4.1.3i}
q^{\frac{1}{2}b}E_{(a,b)}X_1=q^{-\frac{1}{2}(a+1)b}X_{1}^{a+1}X_{2}^{b}=E_{(a+1,b)}.
\end{equation}

(ii) If $a\geq0,b<0$, then
\begin{equation}\label{equ4.1.3ii}
q^{\frac{1}{2}b}E_{(a,b)}X_1=q^{-\frac{1}{2}(a+1)b}X_{1}^{a+1}X_{0}^{-b}=E_{(a+1,b)}.
\end{equation}

(iii) If $a<0,b\geq0$, then
\begin{equation}\label{equ4.1.3iii}
q^{\frac{1}{2}b}E_{(a,b)}X_1-E_{(a+1,b)}=q^{2a}E_{(a+1,b+4)}.
\end{equation}

(iv) If $a=-1,b<0$, then
$$
q^{\frac{1}{2}b}E_{(-1,b)}X_1-E_{(0,b)}=q^{-2}X_{2}^{4}X_{0}^{-b}=q^{-2}E_{(0,4)}X_{0}^{-b}.
$$
When $a=-1,b=-1$, we have that
\begin{equation}\label{equ4.1.3iv1}
q^{-\frac{1}{2}}E_{(-1,-1)}X_1-E_{(0,-1)}=q^{-2}E_{(0,3)}+q^{-4}E_{(1,3)}.
\end{equation}
When $a=-1,b=-2$, we have that
\begin{equation}\label{equ4.1.3iv2}
q^{-1}E_{(-1,-2)}X_1-E_{(0,-2)}=q^{-2}E_{(0,2)}+(q^{-\frac{9}{2}}+q^{-\frac{7}{2}})E_{(1,2)}+ q^{-6}E_{(2,2)}
\end{equation}
When $a=-1,b=-3$, we have that
\begin{align}\label{equ4.1.3iv3}
&q^{-\frac{3}{2}}E_{(-1,-3)}X_1-E_{(0,-3)} \nonumber \\
=&q^{-2}E_{(0,1)} +(q^{-5}+q^{-4}+q^{-3})E_{(1,1)}+(q^{-7}+q^{-6}+q^{-5})E_{(2,1)} +q^{-8}E_{(3,1)}.
\end{align}
When $a=-1,b\leq-4$, we have that
\begin{align}\label{equ4.1.3iv4}
&q^{\frac{1}{2}b}E_{(-1,b)}X_1-E_{(0,b)}\nonumber\\
=&q^{-2}E_{(0,b+4)} +q^{\frac{1}{2}b}(q^{-\frac{7}{2}}+q^{-\frac{5}{2}}+q^{-\frac{3}{2}}+q^{-\frac{1}{2}})E_{(1,b+4)} \nonumber\\
&+q^{b}(q^{-4} +q^{-3}+2q^{-2}+q^{-1}+1)E_{(2,b+4)}+ q^{\frac{3}{2}b}(q^{-\frac{7}{2}} +q^{-\frac{5}{2}} +q^{-\frac{3}{2}}
+q^{-\frac{1}{2}})E_{(3,b+4)} \nonumber\\
&+q^{2b-2}E_{(4,b+4)}.
\end{align}

(v) If $a\leq-2,b<0$, then
$$
q^{\frac{1}{2}b}E_{(a,b)}X_1-E_{(a+1,b)}=q^{-\frac{1}{2}(ab+b+4)}X_{3}^{-a-1}X_{2}^{4}X_{0}^{-b} =q^{2a-\frac{1}{2}(a+1)b}E_{(a+1,4)}X_{0}^{-b}.
$$
When $a=-2,b=-1$, we have that
$$
q^{-\frac{1}{2}}E_{(-2,-1)}X_1-E_{(-1,-1)}=q^{-\frac{5}{2}}X_3X_{2}^{4}X_0=q^{-6}E_{(0,3)} +q^{-8}E_{(0,7)}+q^{-4}E_{(-1,3)}.
$$
When $a\leq-2$, by repeatedly using (\ref{equ4.1.1iii}) and (\ref{equ4.1.1iv}), we know that $q^{2a}E_{(a+1,4+b)}$ is a term of $q^{2a-\frac{1}{2}(a+1)b}E_{(a+1,4)}X_{0}^{-b}$. We only need to consider the expansion of
$$
q^{2a-\frac{1}{2}(a+1)b}E_{(a+1,4)}X_{0}^{-b} -q^{2a}E_{(a+1,4+b)}
$$
with respect to the standard monomials $E_{(c,d)}$.

Note that
\begin{align}\label{equ4.1.3v}
\left\{
\begin{aligned}
E_{(c,d)}X_0=q^{-\frac{1}{2}c}E_{(c,d-1)}+q^{-\frac{1}{2}(c+d)}E_{(c+1,d-1)},{\hskip 3.3cm}&~\text{if~}c\geq0,d>0;\\
E_{(c,d)}X_0=q^{-\frac{1}{2}c}E_{(c,d-1)},{\hskip 6.85cm} &~\text{if~}c\geq0,d\leq0;\\
E_{(c,d)}X_0=q^{-\frac{1}{2}c}E_{(c,d-1)},{\hskip 6.85cm}&~\text{if~}c<0,d\leq0;\\
E_{(c,d)}X_0=q^{-\frac{1}{2}c}E_{(c,d-1)}+q^{-\frac{1}{2}(c+d)}E_{(c+1,d-1)} +q^{\frac{1}{2}(3c-d)}E_{(c+1,d+3)},&~\text{if~}c<0,d>0.
\end{aligned}
\right.
\end{align}
By applying (\ref{equ4.1.3v}) repeatedly, we have that
$$
q^{2a-\frac{1}{2}(a+1)b}E_{(a+1,4)}X_{0}^{-b}-q^{2a}E_{(a+1,4+b)}\in q^{-\frac{1}{2}}\A_+
$$
and every term $kq^{-\frac{1}{2}\alpha}E_{(c,d)}$ ($k\in\NN\setminus\{0\}$) in $q^{2a-\frac{1}{2}(a+1)b}E_{(a+1,4)}X_{0}^{-b}-q^{2a}E_{(a+1,4+b)}$ satisfies the conditions that $\alpha\geq -4a>0$, $c\geq a+1$ and $d\geq 4+b$. Namely
$$
q^{\frac{1}{2}b}E_{(a,b)}X_1-E_{(a+1,b)}=q^{2a-\frac{1}{2}(a+1)b}E_{(a+1,4)}X_{0}^{-b}\in \bigoplus\limits_{(a^{\prime},b^{\prime})\preceq(a+1,b+4)}q^{-\frac{1}{2}}\ZZ[q^{-\frac{1}{2}}]E_{(a^{\prime},b^{\prime})}.
$$

(4) (i) If $a>0,b>0$, then
\begin{equation}\label{equ4.1.4i}
q^{-\frac{1}{2}a}X_4E_{(a,b)}-E_{(a,b-1)}=q^{-\frac{1}{2}b}E_{(a-1,b-1)}+q^{-\frac{1}{2}(4a+b)}E_{(a-1,b+3)}.
\end{equation}

(ii) If $a\leq0,b>0$, then
\begin{equation}\label{equ4.1.4ii}
q^{-\frac{1}{2}a}X_4E_{(a,b)}-E_{(a,b-1)}=q^{-\frac{1}{2}b}E_{(a-1,b-1)}.
\end{equation}

(5) (i) If $a>0,b=0$, then
\begin{equation}\label{equ4.1.5i}
q^{-\frac{1}{2}a}X_4E_{(a,0)}-E_{(a-1,-1)}=q^{-2a}E_{(a-1,3)}.
\end{equation}

(ii) If $a>0,b=-1$, then
\begin{equation}\label{equ4.1.5ii}
q^{-\frac{1}{2}(a+1)}X_4E_{(a,-1)}-E_{(a-1,-2)}=q^{-2a}E_{(a-1,2)}+q^{-\frac{1}{2}(4a+3)}E_{(a,2)}.
\end{equation}

(iii) If $a>0,b=-2$, then
\begin{equation}\label{equ4.1.5iii}
q^{-\frac{1}{2}(a+2)}X_4E_{(a,-2)}-E_{(a-1,-3)}=q^{-2a}E_{(a-1,1)}+q^{-2a}(q^{-2}+q^{-1})E_{(a,1)}+q^{-2a-3}E_{(a+1,1)}.
\end{equation}

(iv) If $a>0,b\leq-3$, then
\begin{align}\label{equ4.1.5iv}
&q^{-\frac{1}{2}(a-b)}X_4E_{(a,b)}-E_{(a-1,b-1)}\nonumber\\
=&q^{-2a}E_{(a-1,b+3)}+q^{-\frac{1}{2}(4a-b-3)}(q^{-\frac{5}{2}}+q^{-\frac{3}{2}}+q^{-\frac{1}{2}})E_{(a,b+3)}\nonumber \\
&+q^{-\frac{1}{2}(4a-2b-3)}(q^{-\frac{5}{2}}+q^{-\frac{3}{2}}+q^{-\frac{1}{2}})E_{(a+1,b+3)} +q^{-\frac{1}{2}(4a-3b)}E_{(a+2,b+3)}.
\end{align}

(6) Through a direct calculation, we have that
\begin{align*}
&X_4E_{(0,0)}=X_4=E_{(-1,-1)}-q^{-2}E_{(0,3)},\\
&q^{\frac{1}{2}}X_4E_{(-1,0)}=q^{\frac{1}{2}}X_4X_3=E_{(-2,-1)}-q^{-4}E_{(-1,3)},\\
&q^{-\frac{1}{2}}X_4E_{(0,-1)}=q^{-\frac{1}{2}}X_4X_0=E_{(-1,-2)}-q^{-\frac{5}{2}}E_{(0,2)}-q^{-4}E_{(1,2)}.
\end{align*}
We will prove the statement by induction on both value $a$ and $b$. For fixed $a\leq0,b\leq0$, assume that
$$
q^{-\frac{1}{2}(a-b)}X_4E_{(a,b)}-E_{(a-1,b-1)}=:\Delta\in \bigoplus\limits_{(a^{\prime},b^{\prime})\preceq(a,b)}q^{-\frac{1}{2}}\ZZ[q^{-\frac{1}{2}}]E_{(a^{\prime},b^{\prime})}.
$$

Note that
\begin{align*}
&q^{-\frac{1}{2}(a-1-b)}X_4E_{(a-1,b)}=q^{-\frac{1}{2}(1-b)}X_3 q^{-\frac{1}{2}(a-b)} X_4E_{(a,b)} \\ =&q^{-\frac{1}{2}(1-b)}X_3E_{(a-1,b-1)}+q^{-\frac{1}{2}(1-b)}X_3\Delta
=E_{(a-2,b-1)}+q^{-\frac{1}{2}(1-b)}X_3\Delta
\end{align*}
and
\begin{align*}
&q^{-\frac{1}{2}(a-b+1)}X_4E_{(a,b-1)}=q^{-\frac{1}{2}(a-b)}X_4E_{(a,b)} q^{-\frac{1}{2}(1-a)}X_0\\
=&q^{-\frac{1}{2}(1-a)}E_{(a-1,b-1)}X_0+q^{-\frac{1}{2}(1-a)}\Delta X_0
=E_{(a-1,b-2)}+q^{-\frac{1}{2}(1-a)}\Delta X_0.
\end{align*}
Let $kq^{-\frac{1}{2}\alpha}E_{(c,d)}$ ($k\in\ZZ\setminus\{0\}$) be a term in $\Delta$. Note that $\alpha+1-b+d\geq\alpha+1>0$. By Lemma \ref{lem4.1}(2), we have that
\begin{align*}
&q^{-\frac{1}{2}(1-b)}X_3 q^{-\frac{1}{2}\alpha}E_{(c,d)}=q^{-\frac{1}{2}(\alpha+1-b+d)}(q^{\frac{1}{2}d}X_3E_{(c,d)}) \\
\in& q^{-\frac{1}{2}(\alpha+1-b+d)} ( E_{(c-1,d)}+\bigoplus\limits_{(a^{\prime},b^{\prime})\preceq(c-1,d+4)}q^{-\frac{1}{2}}\ZZ[q^{-\frac{1}{2}}] E_{(a^{\prime},b^{\prime})}) \subseteq \bigoplus\limits_{(a^{\prime},b^{\prime})\preceq(c-1,d)}q^{-\frac{1}{2}}\ZZ[q^{-\frac{1}{2}}]E_{(a^{\prime},b^{\prime})}.
\end{align*}
It follows that
$$
q^{-\frac{1}{2}(a-1-b)}X_4E_{(a-1,b)}-E_{(a-2,b-1)} \in \bigoplus\limits_{(a^{\prime},b^{\prime})\preceq(a-1,b)}q^{-\frac{1}{2}}\ZZ[q^{-\frac{1}{2}}]E_{(a^{\prime},b^{\prime})}.
$$

Note that $\alpha+c+1-a\geq\alpha+1>0$. By Lemma \ref{lem4.1}(1), we have that
\begin{align*}
&q^{-\frac{1}{2}\alpha}E_{(c,d)} q^{-\frac{1}{2}(1-a)}X_0=q^{-\frac{1}{2}(\alpha+c+1-a)}(q^{\frac{1}{2}c}E_{(c,d)}X_0)\\
\in& q^{-\frac{1}{2}(\alpha+c+1-a)}(E_{(c,d-1)}+ \bigoplus\limits_{(a^{\prime},b^{\prime})\preceq(c+1,d-1)}q^{-\frac{1}{2}}\ZZ[q^{-\frac{1}{2}}] E_{(a^{\prime},b^{\prime})})\subseteq \bigoplus\limits_{(a^{\prime},b^{\prime})\preceq(c,d-1)}q^{-\frac{1}{2}}\ZZ[q^{-\frac{1}{2}}]E_{(a^{\prime},b^{\prime})}.
\end{align*}
Thus
$$
q^{-\frac{1}{2}(a-b+1)}X_4E_{(a,b-1)}-E_{(a-1,b-2)} \in \bigoplus\limits_{(a^{\prime},b^{\prime})\preceq(a,b-1)}q^{-\frac{1}{2}}\ZZ[q^{-\frac{1}{2}}]E_{(a^{\prime},b^{\prime})}.
$$
The proof is completed.

\end{proof}

\begin{lemma}\label{lemofX3}
Let $a,b\in\ZZ_{<0}$ and $kq^{-\frac{1}{2}\alpha}E_{(c,d)}$ ($k\in\NN\setminus\{0\}$) be a term in $q^{\frac{1}{2}b}E_{(a,b)}X_1-E_{(a+1,b)}$. Then we have that $c\leq a-b+1$ if $d\geq0$ and $c-d\leq a-b+1$ if $d<0$.
\end{lemma}
\begin{proof}
By (\ref{equ4.1.3iv1}), (\ref{equ4.1.3iv2}), (\ref{equ4.1.3iv3}) and (\ref{equ4.1.3iv4}), it follows that $(c,d)\in\{(0,3),(1,3)\}$ if $a=-1$, $b=-1$; $(c,d)\in\{(0,2),(1,2),(2,2)\}$ if $a=-1$, $b=-2$; $(c,d)\in\{(0,1),(1,1),(2,1), (3,1)\}$ if $a=-1$ and $b=-3$; $(c,d)\in \{(0,b+4),(1,b+4),(2,b+4),(3,b+4),(4,b+4)\}$ if $a=-1$ and $b\leq-4$.

When $a\leq -2$ and $b<0$, $q^{\frac{1}{2}b}E_{(a,b)}X_1-E_{(a+1,b)}=q^{2a-\frac{1}{2}(a+1)b}E_{(a+1,4)}X_{0}^{-b}$. By (\ref{equ4.1.3v}), it follows that $c\in\{a+1,\ldots,a-b+1\}$, and we obverse that there exists some nonnegative integer $n\leq -a-1$ such that
$$
d=4+3n-(-b-n)=4+4n+b\leq b-4a.
$$
If $d=4+4n+b\geq0$, then $c\leq (a+1)+n+(-b-n)=a-b+1.$
If $d=4+4n+b<0$, then $c\leq (a+1)+n+(4+3n)=a+5+4n$, it follows that $c-d\leq a-b+1$. Thus $c\leq a-b+1$ if $d\geq0$ and $c-d\leq a-b+1$ if $d<0$.

\end{proof}

By using (\ref{equ4.1.4i}) and (\ref{equ4.1.4ii}), we obtain the following lemma.
\begin{lemma}\label{lem4.3+}
For $a\in\ZZ$ and $b\in\ZZ_{>0}$, every term $q^{-\frac{1}{2}\alpha}E_{(c,d)}$ in $q^{-\frac{1}{2}a}X_4E_{(a,b)}-E_{(a,b-1)}$ satisfies that $c=a-1$ and $d\geq b-1\geq0$.
\end{lemma}

\begin{lemma}\label{lem4.3}
For $a\in\ZZ_{>0}$ (respectively, $a\in\ZZ_{\leq0}$) and $b\in\ZZ_{\leq0}$, every term $kq^{-\frac{1}{2}\alpha}E_{(c,d)}$ ($k\in\ZZ\setminus\{0\}$) in $q^{-\frac{1}{2}(a-b)}X_4E_{(a,b)}-E_{(a-1,b-1)}$ satisfies that $c\geq a-1$ (respectively, $c\geq a$), $c\leq a-b$ if $d\geq0$ and $c-d\leq a-b$ if $d<0$.
\end{lemma}
\begin{proof}
(1) When $a>0$, by using the identities (\ref{equ4.1.5i}), (\ref{equ4.1.5ii}), (\ref{equ4.1.5iii}) and (\ref{equ4.1.5iv}), we have that $(c,d)=(a-1,3)$ if $b=0$; $(c,d)\in\{(a-1,2),(a,2)\}$ if $b=-1$; $(c,d)\in\{(a-1,1),(a,1),(a+1,1)\}$ if $b=-2$; $(c,d)\in\{(a-1,b+3),(a,b+3),(a+1,b+3),(a+2,b+3)\}$ if $b\leq-3$. It is easy to see that $c\geq a-1$, $c<a-b$ if $d\geq0$ and $c-d< a-b$ if $d<0$.

(2) When $a\leq0$, we will prove the statement by induction. By Lemma \ref{lem4.1}(6), it follows that $c\geq a$. When $a=0$, note that
\begin{align*}
&q^{-\frac{1}{2}}X_4E_{(0,-1)}=q^{-\frac{1}{2}}X_4X_0=E_{(-1,-2)}-q^{-\frac{5}{2}}E_{(0,2)}-q^{-4}E_{(1,2)},\\
&q^{-1}X_4E_{(0,-2)}=q^{-1}X_4X_{0}^{2}=E_{(-1,-3)}-q^{-3}E_{(0,1)}-(q^{-5}+q^{-4})E_{(1,1)}-q^{-6}E_{(2,1)},
\end{align*}
and
\begin{align*}
q^{-\frac{1}{2}n}X_4E_{(0,-n)}=&E_{(-1,-n-1)}-q^{-\frac{1}{2}(n+4)}E_{(0,3-n)}-q^{-n}(q^{-3}+q^{-2}+q^{-1})E_{(1,3-n)} \nonumber\\
&-q^{-\frac{3}{2}n}(q^{-3}+q^{-2}+q^{-1})E_{(2,3-n)} -q^{-2n-2}E_{(3,3-n)}
\end{align*}
for $n\geq3$. It is easy to see that $c\leq n= a-b$ if $d\geq0$, $c-d\leq a-b$ if $d<0$.

Note that $q^{\frac{1}{2}}X_4E_{(-1,0)}=E_{(-2,-1)}-q^{-4}E_{(-1,3)}$. When $a=-m\leq-1$ and $b=0$, we have that
\begin{align*}
&q^{-\frac{1}{2}m}X_4E_{(-m,0)}=q^{-\frac{1}{2}(m-1)}X_{3}^{m-1}(q^{-\frac{1}{2}}X_3X_4) \\ =&q^{-\frac{1}{2}(m-1)}X_{3}^{m-1}E_{(-2,-1)}-q^{-\frac{1}{2}(7+m)}X_{3}^{m-1}E_{(-1,3)} \\
=&E_{(-m-1,-1)}-q^{-2(1+m)}E_{(-m,3)}.
\end{align*}

From now on, we assume that $a,b\leq-1$ for the rest of the proof. When $a=b=-1$, we have that
$$
X_4E_{(-1,-1)}=q^{-1}X_4X_3X_0=E_{(-2,-2)}-q^{-\frac{9}{2}}E_{(-1,2)}-q^{-6}E_{(0,2)}-q^{-8}E_{(0,6)}.
$$
When $a=-m,b=-1$, we have that
$$
q^{-\frac{1-m}{2}}X_4E_{(-m,-1)}=E_{(-m-1,-2)}-q^{-\frac{4m+5}{2}}E_{(-m,2)} -q^{-2m-4}E_{(1-m,2)}-q^{-4m-4}E_{(1-m,6)}.
$$
For $a=-m\in\ZZ_{<0}$ and $b=-n\in\ZZ_{<0}$, assume that every term $kq^{-\frac{1}{2}\alpha}E_{(c,d)}$ in
$$
\nabla:=q^{-\frac{n-m}{2}}X_4E_{(-m,-n)}-E_{(-m-1,-n-1)}
$$
satisfies that $c\leq n-m$ if $d\leq0$ and $c-d\leq n-m$ if $d<0$.

When $a=-m$ and $b=-n-1$, we have that
\begin{align*}
&q^{-\frac{1}{2}(n+1-m)}X_4E_{(-m,-n-1)}=q^{-\frac{1}{2}(m+1)}(q^{-\frac{1}{2}(n-m)}X_4E_{(-m,-n)})X_0\\
=&q^{-\frac{1}{2}(m+1)}(E_{(-m-1,-n-1)}+\nabla)X_0=E_{(-m-1,-n-2)}+q^{-\frac{1}{2}(m+1)}\nabla X_0.
\end{align*}

Let $k^\prime q^{-\frac{1}{2}\alpha^{\prime}}E_{(c^{\prime},d^{\prime})}$ ($k^\prime\in\ZZ\setminus\{0\}$) be a term in $q^{-\frac{m+1}{2}}(kq^{-\frac{1}{2}\alpha}E_{(c,d)}) X_0$. When $c\geq0,d>0$, it follows that $(c^{\prime},d^{\prime})\in\{(c,d-1),(c+1,d-1)\}$ by (\ref{equ4.1.1i}), i.e., $d^{\prime}\geq0$ and $c\leq c^{\prime}\leq c+1\leq n+1-m$. If $d\leq0$ then $(c^{\prime},d^{\prime})=(c,d-1)$ by (\ref{equ4.1.1ii}) and (\ref{equ4.1.1iii}), i.e., $c^\prime=c$, $d^{\prime}<0$ and $c^{\prime}-d^{\prime}= c-d+1\leq n-m+1$. If $c<0,d>0$, then $(c^{\prime},d^{\prime})\in\{(c,d-1),(c+1,d+3),(c+1,d-1)\}$ by (\ref{equ4.1.1iv}), i.e., $d^{\prime}\geq0$ and $c\leq c^{\prime}\leq c+1\leq n-m+1$.

The proof is completed.
\end{proof}

By \cite[Theorem 3.1, Proposition 4.1]{BZ2014}, we have the following lemma.
\begin{lemma}\label{lemvarphi}
For each $(a,b)\in\ZZ^2$, we have that
$$
E^{\prime}_{(a,b)}-E_{\varphi(a,b)}\in q^{-\frac{1}{2}}\A_+,
$$
where $\varphi:\ZZ^2\rightarrow\ZZ^2$ defined by $\varphi(a,b)=(a,-4[-a]_+-b)$ is a bijection.
\end{lemma}

Similarly, let $\psi:\ZZ^2\rightarrow\ZZ^2$ be a map defined by $\psi(a,b)=(-a-[-b]_+,b)$. It is easy to see that $\psi$ is a bijection.
\begin{lemma}\label{lempsi}
For each $(a,b)\in\ZZ^{2}$, we have that
$$
\mu_1E_{(a,b)}-E_{\psi(a,b)}\in q^{-\frac{1}{2}}\A_+.
$$
\end{lemma}
\begin{proof}
Let $a_1,a_2\in\ZZ_{\geq0}$, it is clear that
$$
\mu_1E_{(a_1,a_2)}=q^{-\frac{1}{2}a_1a_2}X_{2}^{a_2}X_{3}^{a_1}=E_{(-a_1,a_2)}\text{~and~}
\mu_1E_{(-a_1,a_2)}=q^{-\frac{1}{2}a_1a_2}X_{1}^{a_1}X_{2}^{a_2}=E_{(a_1,a_2)}.
$$
We need only consider $\mu_1E_{(a_1,-a_2)}-E_{(-a_1-a_2,-a_2)}$ and $\mu_1E_{(-a_1,-a_2)}-E_{(a_1-a_2,-a_2)}$.

(1) When $a_1=1,a_2=0$, we have that $\mu_1E_{(1,0)}=X_3=E_{(-1,0)}$. When $a_1=0,a_2=1$,
$$
\mu_1E_{(0,-1)}-E_{(-1,-1)}=X_4-q^{-\frac{1}{2}}X_3X_0=-q^{-2}X_{2}^{3}=-q^{-2}E_{(0,3)}.
$$
When $a_1=1,a_2=1$, we have that
$$
\mu_1E_{(1,-1)}-E_{(-2,-1)}=q^{-\frac{1}{2}}X_3X_4-q^{-1}X_{3}^{2}X_0=-q^{-\frac{11}{2}}X_{2}^{3}X_3=-q^{-4}E_{(-1,3)}.
$$

We will prove the statement by induction on both $a_1$ and $a_2$. Assume that
\begin{align*}
&\mu_1E_{(a_1,-a_2)}-E_{(-a_1-a_2,-a_2)}\\
=&q^{-\frac{1}{2}a_1a_2}X_{3}^{a_1}X_{4}^{a_2}-q^{-\frac{1}{2}(a_1+a_2)a_2}X_{3}^{a_1+a_2}X_{0}^{a_2}
=:C_1\in \bigoplus\limits_{(a^{\prime},b^{\prime})\preceq(-a_1-a_2,-a_2)}q^{-\frac{1}{2}} \ZZ[q^{-\frac{1}{2}}]E_{(a^{\prime},b^{\prime})},
\end{align*}
i.e., every term $kq^{-\frac{1}{2}\alpha}E_{(c,d)}$ ($k\in\ZZ\setminus\{0\}$) in $C_1$ satisfies $[-c]_+\leq a_1+a_2$ and $[-d]_+\leq a_2$.

(i) By Lemma \ref{lem4.1}(2), we have that
$$
\mu_1E_{(a_1+1,-a_2)}=q^{-\frac{1}{2}a_2}X_3E_{(-a_1-a_2,-a_2)}+q^{-\frac{1}{2} a_2}X_3C_1 =E_{(-a_1-a_2-1,-a_2)}+q^{-\frac{1}{2}a_2}X_3C_1
$$
and $q^{-\frac{1}{2}(\alpha+a_2)}X_3E_{(c,d)}=q^{-\frac{1}{2}(\alpha+a_2+d)}(q^{\frac{1}{2}d}X_3E_{(c,d)})$. Note that  $\alpha+a_2+d\geq\alpha>0$, by Lemma \ref{lem4.1}(2), we have that $q^{-\frac{1}{2}(\alpha+a_2)}X_3E_{(c,d)}\in \bigoplus\limits_{(a^{\prime},b^{\prime})\preceq(c-1,d)}q^{-\frac{1}{2}}\ZZ[q^{-\frac{1}{2}}]E_{(a^{\prime},b^{\prime})}$.
Thus
$$
\mu_1E_{(a_1+1,-a_2)}-E_{(-a_1-a_2-1,-a_2)}\in \bigoplus\limits_{(a^{\prime},b^{\prime})\preceq(-a_1-a_2-1,-a_2)}q^{-\frac{1}{2}}\ZZ[q^{-\frac{1}{2}}] E_{(a^{\prime},b^{\prime})}.
$$

(ii) Note that
$$
\mu_1E_{(a_1,-a_2-1)}=q^{\frac{a_1}{2}}X_4(q^{-\frac{1}{2}a_1a_2}X_{3}^{a_1}X_{4}^{a_2}) =q^{\frac{a_1}{2}}X_4E_{(-a_1-a_2,-a_2)}+q^{\frac{a_1}{2}}X_4C_1.
$$
By Lemma \ref{lem4.1}(6),
$$
q^{\frac{a_1}{2}}X_4E_{(-a_1-a_2,-a_2)}\in E_{(-a_1-a_2-1,-a_2-1)}+ \bigoplus\limits_{(a^{\prime},b^{\prime})\preceq(-a_1-a_2,-a_2)}q^{-\frac{1}{2}}\ZZ[q^{-\frac{1}{2}}] E_{(a^{\prime},b^{\prime})}.
$$

When $d\geq0$, we have that $q^{-\frac{1}{2}(\alpha-a_1)}X_4E_{(c,d)}=q^{-\frac{1}{2}(\alpha-a_1-c)}(q^{-\frac{1}{2}c}X_4E_{(c,d)})$. When $d<0$, we have that $q^{-\frac{1}{2}(\alpha-a_1)}X_4E_{(c,d)}=q^{-\frac{1}{2}(\alpha-a_1-c+d)}(q^{-\frac{1}{2}(c-d)}X_4E_{(c,d)})$. By Lemma \ref{lem4.3}, we know that $c\leq -a_1$ if $d\geq0$, and $c-d\leq -a_1$ if $d<0$. Thus $\alpha-a_1-c\geq \alpha>0$ if $d\geq0$, and $\alpha-a_1-c+d\geq \alpha>0$ if $d<0$. By using the induction hypothesis and Lemma \ref{lem4.1}(4), (5), (6), we have that
$$
q^{\frac{1}{2}a_1}X_4C_1\in \bigoplus\limits_{(a^{\prime},b^{\prime})\preceq(-a_1-a_2-1,-a_2-1)}q^{-\frac{1}{2}}\ZZ[q^{-\frac{1}{2}}] E_{(a^{\prime},b^{\prime})}.
$$
Hence $\mu_1E_{(a_1,-a_2-1)}-E_{(-a_1-a_2-1,-a_2-1)}\in \bigoplus\limits_{(a^{\prime},b^{\prime})\preceq(-a_1-a_2-1,-a_2-1)}q^{-\frac{1}{2}}\ZZ[q^{-\frac{1}{2}}] E_{(a^{\prime},b^{\prime})}$.

(2) When $a_1=0,a_2>0$, we have that
$$
\mu_1E_{(0,-a_2)}-E_{(-a_2,-a_2)}\in \bigoplus\limits_{(a^{\prime},b^{\prime})\preceq (-a_2,-a_2)}q^{-\frac{1}{2}}\ZZ[q^{-\frac{1}{2}}] E_{(a^{\prime},b^{\prime})}
$$
by (1). When $a_1>0,a_2=0$, we have that $\mu_1E_{(-a_1,0)}=X_{1}^{a_1}=E_{(a_1,0)}$.

When $a_1=a_2=1$,  we have that $\mu_1E_{(-1,-1)}=X_0+q^{-2}X_{2}^{3}=E_{(0,-1)}+q^{-2}E_{(0,3)}$.

When $a_1=1,a_2=2$, we have that $\mu_1E_{(-1,-2)}-E_{(-1,-2)}=q^{-4}E_{(-1,2)}-q^{-4}E_{(1,2)}$.

When $a_1=2,a_2=1$, we have that $\mu_1E_{(-2,-1)}-E_{(1,-1)}=q^{-4}E_{(1,3)}$.

When $a_1=2,a_2=2$, we have that
$$
\mu_1E_{(-2,-2)}-E_{(0,-2)}=(q^{-6}+q^{-2})E_{(0,2)}+q^{-8}E_{(0,6)} +(q^{-\frac{9}{2}}+q^{-\frac{7}{2}})E_{(1,2)}.
$$

When $a_1=3,a_2=1$, we have that $\mu_1E_{(-3,-1)}-E_{(2,-1)}=q^{-6}E_{(2,3)}$.

When $a_1=3,a_2=2$, we have that
$$
\mu_1E_{(-3,-2)}-E_{(1,-2)}=(q^{-8}+q^{-4})E_{(1,2)}+q^{-12}E_{(1,6)}+(q^{-\frac{13}{2}}+q^{-\frac{11}{2}})E_{(2,2)}.
$$

We will proceed the proof by induction on both $a_1$ and $a_2$.

(i) When $a_1\geq a_2\geq1$, we assume that
\begin{align*}
&\mu_1E_{(-a_1,-a_2)}-E_{(a_1-a_2,-a_2)}
=q^{-\frac{1}{2}a_1a_2}X_{4}^{a_2}X_{1}^{a_1}- q^{-\frac{1}{2}(a_1-a_2)a_2}X_{0}^{a_2}X_{1}^{a_1-a_2}\\
=:&C_2\in \bigoplus\limits_{(a^{\prime},b^{\prime})\preceq(0,-a_2)}q^{-\frac{1}{2}}\ZZ[q^{-\frac{1}{2}}]E_{(a^{\prime},b^{\prime})}
\end{align*}
and every term $k_1q^{-\frac{1}{2}\alpha}E_{(c,d)}$ ($k_1\in\ZZ\setminus\{0\}$) in $C_2$ satisfies that $c\geq0$, $c\leq a_1$ if $d\geq0$ and $c-d\leq a_1$ if $d<0$. Note that
\begin{align*}
&\mu_1E_{(-a_1-1,-a_2)}=q^{-\frac{1}{2}a_2}\mu_1E_{(-a_1,-a_2)}X_1 =q^{-\frac{1}{2}a_2}E_{(a_1-a_2,-a_2)}X_1 +q^{-\frac{1}{2}a_2}C_2X_1 \\
=&E_{(a_1+1-a_2,-a_2)}+q^{-\frac{1}{2}a_2}C_2X_1.
\end{align*}
By (\ref{equ4.1.3i}) and (\ref{equ4.1.3ii}), $q^{-\frac{1}{2}(a_2+\alpha)}E_{(c,d)}X_1=q^{-\frac{1}{2}(a_2+\alpha+d)} E_{(c+1,d)}$.

Note that $a_2+\alpha+d\geq\alpha>0$ and $c+1>c\geq0$. Thus
$$
q^{-\frac{1}{2}a_2}C_2X_1\in \bigoplus\limits_{(a^{\prime},b^{\prime})\preceq(0,-a_2)} q^{-\frac{1}{2}}\ZZ[q^{-\frac{1}{2}}]E_{(a^{\prime},b^{\prime})} \subseteq q^{-\frac{1}{2}}\A_+.
$$
For $a_1\geq a_2\geq1$, it follows that
$$
\mu_1E_{(-a_1-1,-a_2)}- E_{(a_1+1-a_2,-a_2)} \in \bigoplus\limits_{(a^{\prime},b^{\prime})\preceq(0,-a_2)}q^{-\frac{1}{2}}\ZZ[q^{-\frac{1}{2}}] E_{(a^{\prime},b^{\prime})} \subseteq q^{-\frac{1}{2}}\A_+
$$
and every term $k_1^{\prime}q^{-\frac{1}{2}\alpha^{\prime}}E_{(c^{\prime},d^{\prime})}=k_1q^{-\frac{1}{2}(a_2+\alpha+d)}E_{(c+1,d)}$ in $\mu_1E_{(-a_1-1,-a_2)}- E_{(a_1+1-a_2,-a_2)}$ satisfies that $c^\prime=c+1>0$, $c^\prime\leq a_1+1$ if $d^\prime=d\geq0$ and $c^\prime-d^\prime\leq a_1+1$ if $d^\prime=d<0$.

(ii) When $a_2>a_1\geq1$, we assume that
\begin{align*}
&\mu_1E_{(-a_1,-a_2)}-E_{(a_1-a_2,-a_2)} \\
=&q^{-\frac{1}{2}a_1a_2}X_{4}^{a_2}X_{1}^{a_1}- q^{-\frac{1}{2}(a_2-a_1)a_2}X_{3}^{a_2-a_1}X_{0}^{a_2}
=:C_3\in \bigoplus\limits_{(a^{\prime},b^{\prime})\preceq(a_1-a_2,-a_2)}q^{-\frac{1}{2}}\ZZ[q^{-\frac{1}{2}}]E_{(a^{\prime},b^{\prime})}.
\end{align*}
and every term $k_2q^{-\frac{1}{2}\beta}E_{(e,f)}$ ($k_2\in\ZZ\setminus\{0\}$) in $C_3$ satisfies that $e\leq a_1$ if $f\geq0$ and $e-f\leq a_1$ if $f<0$. Note that
\begin{align*}
\mu_1E_{(-a_1-1,-a_2)}=q^{-\frac{1}{2}a_2}\mu_1E_{(-a_1,-a_2)}X_1=q^{-\frac{1}{2}a_2}E_{(a_1-a_2,-a_2)}X_1 +q^{-\frac{1}{2}a_2}C_3X_1.
\end{align*}
By Lemma \ref{lem4.1}(3), we obtain that
$$
q^{-\frac{1}{2}a_2}E_{(a_1-a_2,-a_2)}X_1\in E_{(a_1+1-a_2,-a_2)}+\bigoplus\limits_{(a^{\prime},b^{\prime})\preceq(a_1+1-a_2,4-a_2)}q^{-\frac{1}{2}} \ZZ[q^{-\frac{1}{2}}]E_{(a^{\prime},b^{\prime})}.
$$
Note that $a_2+\beta+f>0$ since $[-f]_+\leq a_2$ and $\beta>0$. Then we have that
\begin{align*}
&q^{-\frac{1}{2}(a_2+\beta)}E_{(e,f)}X_1\in q^{-\frac{1}{2}(a_2+\beta+f)}E_{(e+1,f)} +\bigoplus\limits_{(a^{\prime},b^{\prime})\preceq(e+1,f+4)} q^{-\frac{1}{2}}\ZZ[q^{-\frac{1}{2}}] E_{(a^{\prime},b^{\prime})}\\
\subseteq& \bigoplus\limits_{(a^{\prime},b^{\prime})\preceq(e+1,f)} q^{-\frac{1}{2}}\ZZ[q^{-\frac{1}{2}}]E_{(a^{\prime},b^{\prime})}.
\end{align*}
We obtain that $q^{-\frac{1}{2}a_2}C_3X_1\in \bigoplus\limits_{(a^{\prime},b^{\prime})\preceq(a_1+1-a_2,-a_2)} q^{-\frac{1}{2}}\ZZ[q^{-\frac{1}{2}}]E_{(a^{\prime},b^{\prime})}$.

Let $k_{2}^{\prime}q^{-\frac{1}{2}\beta^\prime}E_{(e^\prime,f^\prime)}$ ($k_2\in\ZZ\setminus\{0\}$) be a term in $\mu_1E_{(-a_1-1,-a_2)}-E_{(a_1-a_2+1,-a_2)}$.

When $k_{2}^{\prime}q^{-\frac{1}{2}\beta^\prime}E_{(e^\prime,f^\prime)}$ is a term in $q^{-\frac{1}{2}a_2}E_{(a_1-a_2,-a_2)}X_1-E_{(a_1+1-a_2,-a_2)}$, we obtain that $e^\prime\leq a_1+1$ if $f^\prime\geq0$ and $e^\prime-f^\prime\leq a_1+1$ if $f^\prime<0$, by Lemma \ref{lemofX3}.

Now we consider the case that $k_{2}^{\prime}q^{-\frac{1}{2}\beta^\prime}E_{(e^\prime,f^\prime)}$ is a term in $k_2q^{-\frac{1}{2}(a_2+\beta)}E_{(e,f)}X_1$, where $k_2q^{-\frac{1}{2}\beta} E_{(e,f)}$ is a term in $C_3$. Note that $a_2+\beta +f\geq\beta>0$ and
$$
q^{-\frac{1}{2}(a_2+\beta)}E_{(e,f)}X_1=q^{-\frac{1}{2}(a_2+\beta+f)}(q^{\frac{1}{2}f}E_{(e,f)}X_1).
$$
When $e\geq0$, we have $(e^{\prime},f^{\prime})=(e+1,f)$ by (\ref{equ4.1.3i}) and (\ref{equ4.1.3ii}). When $e<0$ and $f\geq0$, we have $(e^{\prime},f^{\prime})=(e+1,f)$ or $(e+1,f+4)$ by (\ref{equ4.1.3iii}). When $e<0$ and $f<0$, we obtain that $e^\prime\leq e-f+1\leq a_1+1$ if $f^\prime\geq0$ and $e^\prime-f^\prime\leq e-f+1\leq a_1+1$ if $f^\prime<0$ by Lemma \ref{lemofX3}.

Hence
$$
\mu_1E_{(-a_1-1,-a_2)}-E_{(a_1-a_2+1,-a_2)}\in\bigoplus\limits_{(a^{\prime},b^{\prime})\preceq(a_1-a_2+1,-a_2)} q^{-\frac{1}{2}}\ZZ[q^{-\frac{1}{2}}]E_{(a^{\prime},b^{\prime})}\subseteq q^{-\frac{1}{2}}A_+
$$
for $a_2>a_1\geq1$, and every term $k_{2}^{\prime}q^{-\frac{1}{2}\beta^\prime}E_{(e^\prime,f^\prime)}$ in $\mu_1E_{(-a_1-1,-a_2)}-E_{(a_1+1-a_2,-a_2)}$ satisfies that $e^\prime\leq a_1+1$ if $f^\prime\geq0$, $e^\prime-f^\prime\leq a_1+1$ if $f^\prime<0$.

(iii) When $a_1>a_2\geq1$, we assume that
\begin{align*}
\mu_1E_{(-a_1,-a_2)}-E_{(a_1-a_2,-a_2)}=:C_4\in \bigoplus\limits_{(a^{\prime},b^{\prime})\preceq(0,-a_2)}q^{-\frac{1}{2}}\ZZ[q^{-\frac{1}{2}}] E_{(a^{\prime},b^{\prime})}\subseteq q^{-\frac{1}{2}}\A_+
\end{align*}
and every term $k_3q^{-\frac{1}{2}\gamma}E_{(g,h)}$ ($k_3\in\ZZ\setminus\{0\}$) in $C_4$ satisfies that $g\geq a_1-a_2$, $g\leq a_1$ if $h\geq0$ and $g-h\leq a_1$ if $h<0$. We obverse that
$$
\mu_1E_{(-a_1,-a_2-1)}=q^{-\frac{1}{2}a_1}X_4\mu_1E(-a_1,-a_2)=q^{-\frac{1}{2}a_1}X_4E_{(a_1-a_2,-a_2)} +q^{-\frac{1}{2}a_1}X_4C_4
$$
and $q^{-\frac{1}{2}a_1}X_4E_{(a_1-a_2,-a_2)}\in E_{(a_1-a_2-1,-a_2-1)}+ \bigoplus\limits_{(a^{\prime},b^{\prime})\preceq(0,3-a_2)}q^{-\frac{1}{2}}\ZZ[q^{-\frac{1}{2}}]E_{(a^{\prime},b^{\prime})}$ by Lemma \ref{lem4.1}(5).

Similar to the proof of Case (2)(ii), according to Lemma \ref{lem4.3+} and Lemma \ref{lem4.3}, we can prove that every term $k^{\prime}_{3}q^{-\frac{1}{2}\gamma^{\prime}}E_{(g^{\prime},h^{\prime})}$ in $\mu_1E_{(-a_1,-a_2-1)}-E_{(a_1-a_2-1,-a_2-1)}$ satisfies the conditions that $g^\prime\geq a_1-a_2-1\geq0$, $g^{\prime}\leq a_1$ if $h^{\prime}\geq0$, $g^{\prime}-h^{\prime}\leq a_1$ if $h^{\prime}<0$ and
$$
\mu_1E_{(-a_1,-a_2-1)}-E_{(a_1-a_2-1,-a_2-1)}\in \bigoplus\limits_{(a^{\prime},b^{\prime})\preceq(0,-a_2-1)}q^{-\frac{1}{2}}\ZZ[q^{-\frac{1}{2}}] E_{(a^{\prime},b^{\prime})}\subseteq q^{-\frac{1}{2}}\A_+
$$
for $a_1>a_2\geq1$.

(iv) When $a_2\geq a_1\geq1$, we assume that
\begin{align*}
&\mu_1E_{(-a_1,-a_2)}-E_{(a_1-a_2,-a_2)}=q^{-\frac{1}{2}a_1a_2}X_{4}^{a_2}X_{1}^{a_1}- q^{-\frac{1}{2}(a_2-a_1)a_2}X_{3}^{a_2-a_1}X_{0}^{a_2}\\
=:&C_5\in \bigoplus\limits_{(a^{\prime},b^{\prime})\preceq(a_1-a_2,-a_2)}q^{-\frac{1}{2}}\ZZ[q^{-\frac{1}{2}}] E_{(a^{\prime},b^{\prime})}\subseteq q^{-\frac{1}{2}}\A_+
\end{align*}
and every term $k_4q^{-\frac{1}{2}\delta}E_{(i,j)}$ ($k_4\in\ZZ\setminus\{0\}$) in $C_5$ satisfies that $i\leq a_1$ if $j\geq0$ and $i-j\leq a_1$ if $j<0$. Then
$$
\mu_1E_{(-a_1,-a_2-1)}=q^{-\frac{1}{2}a_1}X_4\mu_1E(-a_1,-a_2)=q^{-\frac{1}{2}a_1}X_4E_{(a_1-a_2,-a_2)} +q^{-\frac{1}{2}a_1}X_4C_5
$$
and $q^{-\frac{1}{2}a_1}X_4E_{(a_1-a_2,-a_2)}\in E_{(a_1-a_2-1,-a_2-1)} +\bigoplus\limits_{(a^\prime,b^\prime)\preceq(a_1-a_2,-a_2)} q^{-\frac{1}{2}}\ZZ[q^{-\frac{1}{2}}]E_{(a^\prime,b^\prime)}$ by Lemma \ref{lem4.1}(6).

Let $k_{4}^{\prime}q^{-\frac{1}{2}\delta^\prime}E_{(i^\prime,j^\prime)}$ ($k_{4}^{\prime}\in\ZZ\setminus\{0\}$) be a term in $\mu_1E_{(-a_1,-a_2-1)}-E_{(a_1-a_2-1,-a_2-1)}$. Similar to the proof of Case (2)(ii), according to Lemma \ref{lem4.1}(5), (6), Lemma \ref{lem4.3+} and Lemma \ref{lem4.3}, it follows that
$$
q^{-\frac{1}{2}a_1}X_4C_5\in\bigoplus\limits_{(a^\prime,b^\prime)\preceq(a_1-a_2-1,-a_2-1)} q^{-\frac{1}{2}}\ZZ[q^{-\frac{1}{2}}]E_{(a^\prime,b^\prime)}
$$
and every term $k_{4}^{\prime}q^{-\frac{1}{2}\delta^\prime}E_{(i^\prime,j^\prime)}$ in $q^{-\frac{1}{2}a_1}X_4C_5$ satisfies that $i^\prime\leq a_1$ if $j^{\prime}\geq0$ and $i^\prime-j^\prime\leq a_1$ if $j^{\prime}<0$.

In a summary, we obtain that
$$
\mu_1E_{(-a_1,-a_2-1)}-E_{(a_1-a_2-1,-a_2-1)}\in\bigoplus\limits_{(a^\prime,b^\prime)\preceq(a_1-a_2-1,-a_2-1)} q^{-\frac{1}{2}}\ZZ[q^{-\frac{1}{2}}]E_{(a^\prime,b^\prime)}\subseteq q^{-\frac{1}{2}}\A_+
$$
and every term $k_{4}^{\prime}q^{-\frac{1}{2}\delta^{\prime}}E_{(i^\prime,j^\prime)}$ in $\mu_1E_{(-a_1,-a_2-1)}-E_{(a_1-a_2-1,-a_2-1)}$ satisfies that $i^\prime\leq a_1$ if $j^{\prime}\geq0$ and $i^\prime-j^\prime\leq a_1$ if $j^{\prime}<0$.

This completes the proof of Lemma \ref{lempsi}.
\end{proof}

For $n\in\ZZ$, the notation $\alpha(n)$ is defined by
\begin{align*}
\langle n\rangle\alpha(n)=\left\{
\begin{aligned}
(2-n,2(3-n)),&~\rm{if}~n\geq2; \\
(n,2(n-1)),{\hskip 0.5cm}&~\rm{if}~n\leq1.
\end{aligned}
\right.
\end{align*}
Recall that $C_{(a,b)}$ is the element of the canonical triangular basis for any $a,b\in\ZZ$.

\begin{proposition}\label{propstandardmonomial}
For $n\in\ZZ$ and $a_1,a_2\in\ZZ_{\geq0}$, we have
$$
C_{a_1\alpha(n)+a_2\alpha(n+1)}=q^{-\frac{1}{2}a_1a_2}X_{n}^{a_1}X_{n+1}^{a_2}.
$$
\end{proposition}
\begin{proof}
Through a direct calculation, we obtain that $\sigma_{-2}(\mu_1E_{(a,b)})=E^{\prime}_{(a,b)}$. Recall that $\varphi$ and $\psi$ are bijective. By Lemma \ref{lempsi}, $\mu_1E_{\psi(a,b)}-E_{(a,b)}\in q^{-\frac{1}{2}}\A_+$. Thus $E^{\prime}_{\psi(a,b)}-\sigma_{-2}(E_{(a,b)})\in q^{-\frac{1}{2}}\A_+$. By Lemma \ref{lemvarphi}, we obtain that
\begin{equation}\label{equ4.1}
E_{\varphi\psi(a,b)}-\sigma_{-2}(E_{(a,b)})\in q^{-\frac{1}{2}}\A_+.
\end{equation}

By Lemma \ref{lempsi}, $E^{\prime}_{(a,b)}-\sigma_{-2}(E_{\psi(a,b)})\in q^{-\frac{1}{2}}\A_+$. It follows that $E_{\varphi(a,b)}-\sigma_{-2}(E_{\psi(a,b)})\in q^{-\frac{1}{2}}\A_+$ and $E_{(a,b)}-\sigma_{-2}(E_{\psi\varphi(a,b)})\in q^{-\frac{1}{2}}\A_+$. Hence
\begin{equation}\label{equ4.2}
\sigma_{2}(E_{(a,b)})-E_{\psi\varphi(a,b)}\in q^{-\frac{1}{2}}\A_+.
\end{equation}

Using (\ref{equ4.1}) and (\ref{equ4.2}), it follows that
\begin{equation}\label{equ4.3}
\sigma_{-2}(C_{(a,b)})=C_{\varphi\psi(a,b)}=C_{(-a-[-b]_+,-4[a+[-b]_+]_+-b)}
\end{equation}
and
\begin{equation}\label{equ4.4}
\sigma_{2}(C_{(a,b)})=C_{\psi\varphi(a,b)}=C_{(-a-[4[-a]_++b]_+,-4[-a]_+-b)}.
\end{equation}

According to the definition of the triangular basis, it follows that
\begin{align}\label{equtribasis3}
\left\{
\begin{aligned}
&C_{(a_1,a_2)}=E_{(a_1,a_2)}=q^{-\frac{1}{2}a_1a_2}X_{1}^{a_1}X_{2}^{a_2}; \\
&C_{(a_1,-a_2)}=E_{(a_1,-a_2)}=q^{-\frac{1}{2}a_1a_2}X_{0}^{a_2}X_{1}^{a_1};\\
&C_{(-a_1,a_2)}=E_{(-a_1,a_2)}=q^{-\frac{1}{2}a_1a_2}X_{2}^{a_2}X_{3}^{a_1}
\end{aligned}
\right.
\end{align}
for $a_1,a_2\in\ZZ_{\geq0}$

From (\ref{equ4.3}), (\ref{equ4.4}) and (\ref{equtribasis3}), it is deduced that $C_{a_1\alpha(n)+a_2\alpha(n+1)}=q^{-\frac{1}{2}a_1a_2}X_{n}^{a_1}X_{n+1}^{a_2}$ for $a_1,a_2\in\ZZ_{\geq0}$ and $n\in\ZZ$.

\end{proof}

Now we need only consider $C_{(-n,-2n)}$ for positive integers $n$.
\begin{lemma}\label{lemsn}
If $n$ is a positive integer, then
\begin{align}\label{equsn}
S_{n}(X_\de)
=&q^{-n}X_{n+2}^{\langle n\rangle}X_{0}^{2}-q^{-(n+2)}X_{n+1}^{\langle n+1\rangle}X_1 -q^{-(n+1)}(q^{-\frac{1}{2}}+q^{\frac{1}{2}}) \sum\limits_{k=1}^{\lfloor \frac{n}{2}\rfloor+1}X_{n+3-2k}^{\langle n+1\rangle}\nonumber \\
&-\sum\limits_{k\geq1}kq^{-2k}(q^{-\frac{1}{2}}+q^{\frac{1}{2}})^{2}S_{n-2k}(X_\de).
\end{align}
\end{lemma}
\begin{proof}
We will prove the lemma by induction on $n$. When $n=1$, by \cite[Lemma~3.6]{cds}, we have that
$$
X_\de=q^{-1}X_3X_{0}^2-q^{-3}X_{2}^2X_1-q^{-2}(q^{-\frac{1}{2}}+q^{\frac{1}{2}})X_{2}^2.
$$

When $n$ is even, by the induction hypothesis,
\begin{align*}
&S_{n+1}(X_\de)=X_\de S_{n}(X_\de)-S_{n-1}(X_\de) \\
=& q^{-n}X_\de X_{n+2}^2X_{0}^2-q^{-(n+2)}X_\de X_{n+1}X_1 -q^{-(n+1)}(q^{-\frac{1}{2}}+q^{\frac{1}{2}}) \sum\limits_{k=1}^{\frac{n}{2}+1}X_\de X_{n+3-2k}\\
&-\sum\limits_{k\geq1}kq^{-2k}(q^{-\frac{1}{2}}+q^{\frac{1}{2}})^2X_\de S_{n-2k}(X_\de) -q^{-(n-1)}X_{n+1}X_{0}^{2}+q^{-(n+1)}X_{n}^2X_1\\
&+q^{-n}(q^{-\frac{1}{2}}+q^{\frac{1}{2}})\sum\limits_{k=1}^{\frac{n}{2}}X_{n+2-2k}^{2} +\sum\limits_{k\geq1}kq^{-2k}(q^{-\frac{1}{2}}+q^{\frac{1}{2}})^2S_{n-1-2k}(X_\de).
\end{align*}
A direct calculation shows that
\begin{align*}
&q^{-n}X_\de X_{n+2}^2X_{0}^2=q^{1-n}X_{n+1}X_{0}^2+q^{-n-1}X_{n+3}X_{0}^2+q^{-n}(q^{-\frac{1}{2}}+q^{\frac{1}{2}})X_{0}^2,\\
&q^{-(n+2)}X_\de X_{n+1}X_1=q^{-(n+1)}X_{n}^2X_1+q^{-(n+3)}X_{n+2}^2X_1,\\
&X_\de X_{n+3-2k}=qX_{n+2-2k}+q^{-1}X_{n+4-2k}.
\end{align*}
Note that
$$
q^{-n}(q^{-\frac{1}{2}}+q^{\frac{1}{2}})\sum\limits_{k=1}^{\frac{n}{2}+1} X_{n+2-2k}^2= q^{-n}(q^{-\frac{1}{2}}+q^{\frac{1}{2}})\sum\limits_{k=1}^{\frac{n}{2}} X_{n+2-2k}^2 +q^{-n}(q^{-\frac{1}{2}}+q^{\frac{1}{2}})X_{0}^{2}.
$$
Thus
\begin{align*}
S_{n+1}(X_\de)
=&q^{-(n+1)}X_{n+3}X_{0}^2-q^{-(n+3)}X_{n+2}^2X_1 -q^{-(n+2)}(q^{-\frac{1}{2}}+q^{\frac{1}{2}})\sum\limits_{k=1}^{\frac{n}{2}+1}X_{n+4-2k}^2 \\
&- \sum\limits_{k\geq1}kq^{-2k}(q^{-\frac{1}{2}}+q^{\frac{1}{2}})^2S_{n+1-2k}(X_\de).
\end{align*}

When $n$ is odd, by the induction hypothesis,
\begin{align*}
S_{n+1}(X_\de)=&X_\de S_{n}(X_\de)-S_{n-1}(X_\de) \\
=& q^{-n}X_\de X_{n+2}X_{0}^2-q^{-(n+2)}X_\de X_{n+1}^2X_1 -q^{-(n+1)}(q^{-\frac{1}{2}}+q^{\frac{1}{2}}) \sum\limits_{k=1}^{\frac{n+1}{2}}X_\de X_{n+3-2k}^2\\
&-\sum\limits_{k\geq1}kq^{-2k}(q^{-\frac{1}{2}}+q^{\frac{1}{2}})^2X_\de S_{n-2k}(X_\de) -q^{-(n-1)}X_{n+1}^2X_{0}^{2}+q^{-(n+1)}X_{n}X_1\\
&+q^{-n}(q^{-\frac{1}{2}}+q^{\frac{1}{2}})\sum\limits_{k=1}^{\frac{n+1}{2}}X_{n+2-2k} +\sum\limits_{k\geq1}kq^{-2k}(q^{-\frac{1}{2}}+q^{\frac{1}{2}})^2S_{n-1-2k}(X_\de).
\end{align*}
We have that
\begin{align*}
&q^{-n}X_\de X_{n+2}X_{0}^2=q^{-(n-1)}X_{n+1}^2X_{0}^{2} +q^{-(n+1)}X_{n+3}^2X_{0}^{2},\\
&q^{-(n+2)}X_\de X_{n+1}^2X_1=q^{-(n+1)}X_nX_1+q^{-(n+3)}X_{n+2}X_1+q^{-(n+2)}(q^{-\frac{1}{2}}+q^{\frac{1}{2}})X_1,\\
&q^{-(n+1)}(q^{-\frac{1}{2}}+q^{\frac{1}{2}})\sum\limits_{k=1}^{\frac{n+1}{2}}X_\de X_{n+3-2k}^2 \\ =&q^{-n}(q^{-\frac{1}{2}}+q^{\frac{1}{2}})\sum\limits_{k=1}^{\frac{n+1}{2}}X_{n+2-2k} +q^{-(n+2)}(q^{-\frac{1}{2}}+q^{\frac{1}{2}})\sum\limits_{k=1}^{\frac{n+1}{2}}X_{n+4-2k} +\frac{n+1}{2}q^{-(n+1)} (q^{-\frac{1}{2}}+q^{\frac{1}{2}})^2
\end{align*}
and
$$
\sum\limits_{k\geq1}kq^{-2k}(q^{-\frac{1}{2}}+q^{\frac{1}{2}})^2\big(X_\de S_{n-2k}(X_\de)-S_{n-1-2k}(X_\de)\big) =\sum\limits_{k=1}^{\frac{n-1}{2}}kq^{-2k}(q^{-\frac{1}{2}}+q^{\frac{1}{2}})^2S_{n+1-2k}(X_\de).
$$
Hence we obtain that
\begin{align*}
S_{n+1}(X_\de)=&q^{-(n+1)}X_{n+3}^2X_{0}^2-q^{-(n+3)}X_{n+2}X_1 -q^{-(n+2)}(q^{-\frac{1}{2}}+q^{\frac{1}{2}})\sum\limits_{k=1}^{\frac{n+3}{2}}X_{n+4-2k}\\ &-\sum\limits_{k\geq1}kq^{-2k}(q^{-\frac{1}{2}}+q^{\frac{1}{2}})^2S_{n+1-2k}(X_\de).
\end{align*}
The proof is completed. 
\end{proof}

\begin{lemma}\label{lemsn2}
For $n\in\ZZ_{\geq0}$, we have that $C_{(-n,-2n)}=S_n(X_\delta)$.
\end{lemma}
\begin{proof}
We will prove the statement by induction on $n$. It is trivial for $n=0$. For $n=1$, by Theorem \ref{theorem2}, we have that
$$
E_{(-1,-2)}=X_\de+(q^{-\frac{5}{2}}+q^{-\frac{3}{2}})E_{(0,2)}+q^{-4}E_{(1,2)}.
$$
Therefore $C_{(-1,-2)}=X_\de$. Assume that $C_{(-r,-2r)}=S_r(X_\de)$ for $0\leq r\leq n$. When $n$ is even, we have that
\begin{align*}
S_n(X_\de)=&q^{-n}X_{n+2}^{2}X^{2}_0-q^{-(n+2)}X_{n+1}X_1 -q^{-(n+1)}(q^{-\frac{1}{2}}+q^{\frac{1}{2}})\sum\limits_{k=1}^{\frac{n}{2}+1}X_{n+3-2k} \\ &-\sum\limits_{k\geq1}kq^{-2k}(q^{-\frac{1}{2}}+q^{\frac{1}{2}})^{2}S_{n-2k}(X_\delta).
\end{align*}
By Proposition \ref{propstandardmonomial}, we know that $X_{n+2}^{2}=C_{2\alpha(n+2)}=C_{(-n,2-2n)}$ and $X_{n+1}=C_{\alpha(n+1)}=C_{(1-n,4-2n)}$. By using Lemma \ref{lem4.1}(1), (3) and the condition (\ref{equlusztiglem3}) in Lusztig's Lemma, we obtain that
$$
q^{-n}C_{(-n,2-2n)}X_{0}^2\in q^{-n}E_{(-n,2-2n)}X_{0}^2+q^{-\frac{1}{2}}\A_+
$$
and
$$
q^{-(n+2)}C_{(1-n,4-2n)}X_1\in q^{-(n+2)}E_{(1-n,4-2n)}X_{1}+q^{-\frac{1}{2}}\A_+.
$$
For $n\geq 2$, $q^{-n}E_{(-n,2-2n)}X_{0}^{2}=q^{-\frac{n}{2}}(q^{-\frac{n}{2}}E_{(-n,2-2n)}X_0)X_0 =q^{-\frac{n}{2}}E_{(-n,1-2n)}X_0=E_{(-n,-2n)}$ and $q^{-(n+2)}E_{(1-n,4-2n)}X_1\in q^{-3}(E_{(2-n,4-2n)} +q^{-\frac{1}{2}}\A_+)\subseteq q^{-\frac{1}{2}}\A_+$ by (\ref{equ4.1.1iii}) and Lemma \ref{lem4.1}(3). Hence we have that
$$
q^{-n}X_{n+2}^{2}X^{2}_{0}=q^{-n}C_{(-n,2-2n)}X^{2}_0\in E_{(-n,-2n)}+q^{-\frac{1}{2}}\A_+
$$
and
$$
q^{-(n+2)}X_{n+1}X_1=q^{-(n+2)}C_{(-n,4-2n)}X_1\in q^{-\frac{1}{2}}\A_+.
$$
For $1\leq k\leq \frac{n}{2}+1$, it is clearly that
\begin{align*}
&q^{-(n+1)}(q^{-\frac{1}{2}}+q^{\frac{1}{2}})X_{n+3-2k}\\
=&q^{-(n+1)}(q^{-\frac{1}{2}}+q^{\frac{1}{2}})C_{\alpha(n+3-2k)}\in q^{-(n+1)}(q^{-\frac{1}{2}}+q^{\frac{1}{2}})(E_{(2k-n-1,4k-2n)}+q^{-\frac{1}{2}}\A_+)\subseteq q^{-\frac{1}{2}}\A_+.
\end{align*}
By using the induction hypothesis, we have that $q^{-2k}(q^{-\frac{1}{2}}+q^{\frac{1}{2}})^{2}S_{n-2k}(X_\de)\in q^{-\frac{1}{2}}A_+$ for $1\leq k\leq\frac{n}{2}$. When $n$ is even, we have that
$$
S_{n}(X_\de)\in E_{(-n,-2n)}+q^{-\frac{1}{2}}\A_+.
$$

When $n$ is odd, we have that
\begin{align*}
S_n(X_\de)=&q^{-n}X_{n+2}X^{2}_0-q^{-(n+2)}X_{n+1}^{2}X_1-q^{-(n+1)}(q^{-\frac{1}{2}}+q^{\frac{1}{2}}) \sum\limits_{k=1}^{\frac{n+1}{2}}X_{n+3-2k}^2 \\
&-\sum\limits_{k\geq1}kq^{-2k}(q^{-\frac{1}{2}}+q^{\frac{1}{2}})^{2}S_{n-2k}(X_\de),
\end{align*}
$X_{n+2}=C_{\alpha(n+2)}=C_{(-n,2-2n)}$, $X_{n+1}^2=C_{2\alpha(n+2)}=C_{(1-n,4-2n)}$
 and
$
X^{2}_{n+3-2k}=C_{2\alpha(n+3-2k)}\\ =C_{(2k-n-1,4k-2n)}
$
for $1\leq k\leq \frac{n+1}{2}$. The rest of the proof is quite similar to that given above for the case of $n$ being even and so is omitted.

Since $S_n(X_\de)$ is bar-invariant, we are led to the conclusion that $C_{(-n,-2n)}=S_n(X_\de)$ for $n\in\ZZ_{\geq0}$.
\end{proof}

Proposition \ref{propstandardmonomial} and Lemma \ref{lemsn2} imply the following theorem.
\begin{theorem}
The basis $\mathcal{S}=\{q^{-\frac{1}{2}ab}X^{a}_{m}X^{b}_{m+1}~|~m\in\ZZ,(a,b)\in\ZZ^{2}_{\geq0}\}\cup \{S_{n}(X_\delta)\}$ is the triangular basis in $\A_{q}(1,4)$.
\end{theorem}

\section*{Acknowledgments}

The first draft of the paper was written during X. Chen's visit at Nankai University from May 19 to June 18, 2017. He thanks Nankai University for the hospitality
and for creating ideal working environment. We appreciate Prof. Fan Qin's comments on the relation between the bases obtained in our paper and some other existed bases.

\end{document}